\documentclass[12pt]{amsart}

\usepackage{begnac}
\usepackage[matrix,arrow]{xy}

\DeclareMathOperator{\CBm}{CBm}

\title{Continuous first order logic and local stability}

\author{Ita\"\i{} \textsc{Ben Yaacov}}

\address{Ita\"\i{} \textsc{Ben Yaacov} \\
  Universit\'e de Lyon \\
  Universit\'e Lyon 1 \\
  Institut Camille Jordan, UMR 5208 CNRS\\
  43 boulevard du 11 novembre 1918 \\
  F-69622 Villeurbanne Cedex \\
  France}

\urladdr{http\string://math.univ-lyon1.fr/\textasciitilde begnac/}

\thanks{Research of the first author supported by NSF grant DMS-0500172}

\author{Alexander Usvyatsov}

\address{Alexander Usvyatsov \\
  UCLA Mathematics Department \\
  Box 951555 \\
  Los Angeles, CA 90095-1555 \\
  USA}

\urladdr{http://www.math.ucla.edu/\textasciitilde alexus}

\thanks{The authors would like to thank C.\ Ward Henson for stimulating
  discussions, and Sylvia Carlisle and Eric
  Owiesny for a careful reading of the manuscript}

\date{\today}
\keywords{continuous logic, metric structures, stability, local stability}
\subjclass[2000]{03C90,03C45}

\begin{document}

\begin{abstract}
  We develop continuous first order logic, a variant of the logic
  described in \cite{Chang-Keisler:ContinuousModelTheory}.
  We show that this logic has the same power of expression as the
  framework of open Hausdorff cats, and as such extends Henson's logic
  for Banach space structures.
  We conclude with the development of local stability, for which this
  logic is particularly well-suited.
\end{abstract}

\maketitle

\section*{Introduction}

A common trend in modern model theory is to generalise
model-theoretic notions and tools to frameworks that go beyond that of
first order logic and elementary classes and properties.
In doing this, there is usually a trade-off: the more general the
framework, the weaker the available tools, and one finds oneself many
times trying to play this trade-off, looking for the most general
framework in which a specific argument can be carried through.
The authors admit having committed this sin not once.

The present paper is somewhat different, though: we do present what
seems to be a new framework, or more precisely, a new logic, but in
fact we prove that it is completely equivalent to one that has been
previously defined elsewhere, namely that of (metric) open Hausdorff
cats (see \cite{BenYaacov:Morley}).

Another logic dealing with metric structures is Henson's logic of
positive bounded formulae and approximate satisfaction (see
for example \cite{Henson-Iovino:Ultraproducts}).
Even though Henson's logic was formulated for unbounded Banach space
structures while ours deal with bounded metric structures, it is fair
to say that the two logics are equivalent.
First of all, Henson's
approach makes perfect sense in the bounded setting in which case the
two logics are indeed equivalent.
Banach space structures can (in most cases) be reduced for
model-theoretic purposes to their closed unit ball (see
\fref{exm:BanachUnitBall}).
Moreover, there exists
an \emph{unbounded} variant of
continuous logic which is equivalent with to (a somewhat extended)
Henson's logic for arbitrary (bounded or unbounded) metric structures.
It can be reduced back to continuous logic as studied here (i.e.,
bounded) via the addition of a single point at infinity.
This beyond the scope of the present paper and is discussed in detail
in \cite{BenYaacov:Perturbations}.

Finally, this logic is \emph{almost} a special case of the
continuous first order logic that Chang and Keisler studied in
\cite{Chang-Keisler:ContinuousModelTheory}.
We do differ with their definitions on several
crucial points, where we find they were too general, or not general
enough.
Our logic is a special case in that
instead of allowing any compact Hausdorff space $X$ as a set of
truth values, we find that letting $X$ be the unit interval $[0,1]$
alone is enough.
Indeed, as every compact Hausdorff space embeds into a power of the
interval, there is no loss of generality.
Similarly,
as unit interval admits a natural complete linear ordering,
we may eliminate the plethora of quantifiers present in
\cite{Chang-Keisler:ContinuousModelTheory}, and the arbitrary choices involved, in favour of two
canonical quantifiers, $\sup$ and $\inf$, which are simply the
manifestations in this setting of the classical quantifiers $\forall$ and
$\exists$.
On the other hand, extending Chang and Keisler, we allow the ``equality
symbol'' to take any truth value in $[0,1]$.
Thus, from an \emph{equality} symbol it becomes a \emph{distance}
symbol, allowing us to interpret metric structures
in the modified logic.

However, continuous first order logic has significant advantages over
previous formalisms for metric structures.
To begin with, it is an immediate
generalisation of classical first order logic, more natural and
less technically involved than previous formalisms.
More importantly, it allows us to beat the above-mentioned
trade-off.
Of course, if two logics have the same power of expression, and only
differ in presentation, then an
argument can be carried in one if and only if it can be carried in the
other; but it may still happen that notions which arise naturally from
one of the presentations are more useful, and render clear and obvious
what was obscure with the other one.
This indeed seems to be the case with continuous first order logic,
which further supports our contention that it is indeed the ``true and
correct'' generalisation of classical first order logic to the context
of metric structures arising in analysis.

An example for this, which was part of the original motivation towards
these ideas, is a question by C.\ Ward
Henson, which can be roughly stated as ``how
does one generalise local (i.e., formula-by-formula) stability theory
to the logic of positive bounded formulae?''.
The short answer, as far as we can see, is ``one doesn't.''
The long answer is that positive bounded formulae may not be the
correct analogues of first order formulae for these purposes, whereas
continuous first order formulae are.

\medskip

In \fref{sec:syntax} we define the syntax of continuous first
order logic: signatures, connectives, quantifiers, formulae and
conditions.

In \fref{sec:struct} we define the semantics: pre-structures,
structures, the special role of the metric and truth values.

In \fref{sec:TypesDefPred} we discuss types and definable
predicates.
The family of definable predicates is the completion, in some natural
sense, of the family of continuous formulae.

In \fref{sec:Theories} we discuss continuous first order
theories and some basic properties such as quantifier elimination.
We also compare continuous first order theories with previous notions
such as open Hausdorff cats.

In \fref{sec:Img} we discuss imaginaries as canonical
parameters to formulae and definable predicates.

In \fref{sec:LocTyp} we define $\varphi$-types, i.e., types which only
depend on values of instances of a formula $\varphi$.

In \fref{sec:LocStab} we develop
local stability, answering Henson's question.

In \fref{sec:GlobStab} we show how to deduce the standard global
theory of independence from the local one in a stable theory.

We also have two appendices:

\fref{apdx:InvContMod} contains a remark concerning
an alternative (and useful) presentation of continuity moduli.

\fref{apdx:ModStab} deals with the case of a formula which is
stable in a single model of a theory.

\section{Continuous first order formulae}
\label{sec:syntax}

In classical (first order) logic there are two possible truth values:
``true'', sometimes denoted by $\top$ or $T$, and ``false'',
denoted by $\bot$ or $F$.
Often enough one associates the classical truth values with numerical
values, and the most common choice is probably to assign $T$ the value
$1$ and $F$ the value $0$.
This assignment is not sacred, however, and for our purposes the
opposite assignment, i.e., $T = 0$ and $F = 1$, fits more elegantly.

The basic idea of this paper is to repeat the development of first
order logic with one tiny difference: we replace the \emph{finite}
set of truth values $\{0,1\}$ with the \emph{compact} set $[0,1]$.
Everything else should follow naturally from this modification.
We will refer to the classical framework also as \emph{discrete
logic}, whereas the one we develop here will be referred to as
\emph{continuous logic}.

As in classical logic, a \emph{continuous signature} $\cL$ is a set of
function symbols and predicate symbols, each of which having an
associated arity $n < \omega$.
In an actual continuous structure, the function symbols will be
interpreted as functions from the structure into itself, and the
predicate symbols as
functions to the set of truth values, i.e., the interval $[0,1]$.

For the definition of pure syntax we may restrict ourselves to
non-metric signatures, which are the analogues of classical signatures
without equality.

\begin{dfn}
  A \emph{non-metric continuous signature} consists of a set of
  function symbols and predicate symbols, and for each function symbol
  $f$ or predicate symbol $P$, its arity $n_f < \omega$ or $n_P < \omega$.

  We may also consider multi-sorted signatures, in which case the
  arity of each symbol specify how many arguments are in each sort,
  such that the total is finite, and each function symbol has a target
  sort.
\end{dfn}

Given a continuous signature $\cL$, we define $\cL$-terms and atomic
$\cL$-formulae as usual.
However, since the truth values of
predicates are going to be in $[0,1]$, rather than in $\{0,1\}$, we
need to adapt our connectives and quantifiers accordingly.

Let us start with connectives.
In the discrete setting we use a somewhat fixed set of unary and
binary Boolean connectives, from which we can construct any $n$-ary
Boolean expression.
In other words, any mapping from $\{0,1\}^n \to \{0,1\}$ can be written
using these connectives (otherwise, we would have introduced
additional ones).
By analogy, an $n$-ary continuous connective should be a
\emph{continuous} mapping from $[0,1]^n \to [0,1]$, and we would like to
have a set of connectives with which we can construct every continuous
mapping $[0,1]^n \to [0,1]$, for every $n$.
However, this may be problematic, as continuum many
connectives would give rise to uncountably many formulae even
in a countable signature.
To avoid this anomaly we will content ourselves with
a set of connectives which merely
allows to construct arbitrarily good approximations of every
continuous mapping $[0,1]^n \to [0,1]$.

Common connectives we use, by arity:
\begin{itemize}
\item Constants in $[0,1]$.
\item $\lnot x = 1-x$, and $\frac{x}{2}$.
\item $x\land y = \min \{x,y\}$, $x\lor y = \max\{x,y\}$,
  $x \dotminus y = (x-y)\lor0$,
  $x \dotplus y = (x+y)\land1$,
  $|x-y|$.
\end{itemize}
We can express the non-constant connectives above in terms of the
connectives $\lnot$ and $\dotminus$:
\begin{align*}
  x \land y & = x \dotminus (x \dotminus y) \\
  x \lor y & = \lnot(\lnot x\land\lnot y) \\
  x \dotplus y & = \lnot(\lnot x \dotminus y) \\
  |x - y| & = (x \dotminus y)\lor(y \dotminus x) =
  (x \dotminus y)\dotplus(y \dotminus x)
\end{align*}

The expression $x \dotminus ny$ is a shorthand for
$((x \dotminus y) \dotminus y) \ldots \dotminus y$, $n$ times.
We would also like to point out to the reader that the expression
$x \dotminus y$ is the analogue of the Boolean expression $y \to x$.
For example, the continuous Modus Ponens says that if both $y$ and
$x \dotminus y$ are true, i.e., equal to zero, then so is $x$.

\begin{dfn}
  \begin{enumerate}
  \item A \emph{system of continuous connectives} is a sequence $\cF =
    \{F_n\colon n<\omega\}$ where each $F_n$ is a collection of continuous
    functions from $[0,1]^n$ to $[0,1]$.
  \item We say that a system of continuous connectives $\cF$ is
    \emph{closed} if it satisfies:
    \begin{enumerate}
    \item For all $m < n < \omega$, the projection on the $m$th coordinate
      $\pi_{n,m}\colon [0,1]^n \to [0,1]$ belongs to $F_n$.
    \item Let $f \in F_n$, and $g_0,\ldots,g_{n-1} \in F_m$.
      Then the composition $f\circ(g_0,\ldots,g_{n-1}) \in F_m$.
    \end{enumerate}
    If $\cF$ is any system of continuous connectives, then $\bar \cF$ is
    the closed system it generates.
  \item We say that a closed system of continuous connectives $\cF$ is
    \emph{full} if for every $0 < n < \omega$, the set $F_n$ is dense in
    the set of all continuous functions $\{f\colon [0,1]^n \to [0,1]\}$ in the
    compact-open (i.e., uniform convergence) topology.
    An arbitrary system of continuous connectives $\cF$ is full if $\bar
    \cF$ is.
    \\
    (We exclude $n = 0$ in order to allow full systems of connectives
    without truth constants, i.e., in which $F_0$ is empty.)
  \end{enumerate}
\end{dfn}

\begin{fct}[Stone-Weierstrass Theorem, lattice version]
  \label{fct:LatSW}
  Let $X$ be a compact Hausdorff space containing at least two points,
  $I \subseteq \setR$ an interval,
  and equip $\fA = C(X,I)$ with the uniform convergence topology.
  Let $\fB \subseteq \fA$ be a sub-lattice, such that for every distinct
  $x,y \in X$, $a,b \in I$, and $\varepsilon > 0$, there is $f \in \fB$ such that
  $|f(x)-a|,|f(y)-b| < \varepsilon$.
  Then $\fB$ is dense in $\fA$.
\end{fct}
\begin{proof}
  The proof of this or very similar results should appear in almost
  any analysis textbook.
  We will nonetheless include the proof for completeness.

  Let $f \in \fA$ and $\varepsilon > 0$ be given. For each pair of points $x,
  y \in X$ we can by hypothesis find $g_{x,y} \in \fB$ for which
  $|g_{x,y}(x) -  f(x)|, |g_{x,y}(y) - f(y)|<\varepsilon$.
  (In case $x = y$ we take $g_{x,x} = g_{x,z}$ for any $z \neq x$.)
  The set $V_{x,y} = \{z \in X \colon f(z) - \varepsilon < g_{x,y}(z)\}$ is an open
  neighbourhood of $y$.

  Let us fix $x$.
  The family $\{V_{x,y}\colon y \in X\}$ is an open
  covering of $X$, and admits a finite sub-covering
  $\{V_{x,y_i}\colon i < n\}$.
  Let $g_x = \bigvee_{i<n} g_{x,y_i} \in \fB$.
  Then $f(z)-\varepsilon < g_x(z)$ for all $z \in X$ and
  $|g_x(x) - f(x)|<\varepsilon$.
  Thus $U_x = \{z \in X \colon g_x(z) < f(z) + \varepsilon\}$ is
  an open neighbourhood of $x$.

  Now let $x$ vary.
  The family $\{U_x\colon x\in X\}$ is an open covering of $X$ admitting a
  finite sub-covering $\{U_{x_j} \colon j < m\}$.
  Let $g = \bigwedge_{j<m} g_{x_j}$.
  Then $f(z) - \varepsilon < g(z) < f(z) + \varepsilon$
  for all $z \in X$, i.e., $\|g(z) - f(z)\| < \varepsilon$ as desired.
\end{proof}

It will be more convenient to use the following consequence, which is
analogous to the Stone-Weierstrass characterisation of dense algebras
of functions:
\begin{cor}
  \label{cor:SWCnct}
  Let $X$ be a compact Hausdorff space and let $\fA = C(X,[0,1])$.
  Assume that $\fB \subseteq \fA$ is closed under $\lnot$ and $\dotminus$,
  separates points in $X$ (i.e., for every two distinct
  $x,y \in X$ there is $f \in \fB$ such that $f(x) \neq f(y)$), and satisfies
  either of the following two additional properties:
  \begin{enumerate}
  \item The set $C = \{c \in [0,1]\colon \text{the constant $c$ is in $\fB$}\}$
    is dense in $[0,1]$.
  \item $\fB$ is closed under $x \mapsto \frac{x}{2}$.
  \end{enumerate}
  Then $\fB$ is dense in $\fA$.
\end{cor}
\begin{proof}
  Since $\fB$ is closed under $\lnot$ and $\dotminus$ it is also closed under
  $\lor$ and $\land$, so it is a sub-lattice of $\fA$.

  Since $\fB$ separates points it is in particular non-empty, so we have
  $0 = f \dotminus f \in \fB$ for any $f \in \fB$, whereby $1 = \lnot0 \in \fB$.
  In case $\fB$ is closed under $\frac{x}{2}$ we conclude that
  $1/2^n \in \fB$ for all $n$, and since $\fB$ is also closed under
  $\dotplus$, $\fB$ contains all the dyadic constants in $[0,1]$ which
  are dense in $[0,1]$.
  We may therefore assume that $\fB$ contains a dense set of constants.

  Let $x,y \in X$ be distinct, $a,b \in [0,1]$ and $\varepsilon > 0$.
  Let us first treat the case where $X = [0,1]$ and
  $\id_{[0,1]} \in \fB$.
  Then $x,y \in [0,1]$, and without loss of generality we may assume
  that $x < y$.
  Assume first that $a \geq b$.
  Let $m \in \setN$ be such that $\frac{a}{y-x} < m$, and let
  $f_0(t) = (a \dotminus m(t \dotminus x))\lor b$.
  Then $f_0(x) = a\lor b = a$, and $f_0(y) = 0\lor b = b$.
  Replacing the constants $x,a,b \in [0,1]$ with close enough
  approximations from $C$ we obtain $f \in \fB$ such that
  $|f(x)-a|,|f(y)-b| < \varepsilon$.
  If $a < b$, find $f \in \fB$ such that
  $|f(x)-\lnot a|,|f(y)-\lnot b| < \varepsilon$,
  and then $\lnot f \in \fB$ is as required.

  We now return to the general case where $X$ can be any compact
  Hausdorff space.
  Since $x \neq y$ there is $g \in \fB$ such that
  $g(x) \neq g(y)$, and without loss of generality we may assume that
  $g(x) < g(y)$.
  As in the previous paragraph construct $f\colon [0,1] \to [0,1]$
  such that $|f(g(x))-a|,|f(g(y))-b| < \varepsilon$, and
  observe that $f\circ g \in \fB$.

  We have shown that $\fB$ satisfies the hypotheses of
  \fref{fct:LatSW} and is therefore dense in $\fA$.
\end{proof}

\begin{cor}
  \label{cor:Full}
  Let $C \subseteq [0,1]$ be dense, $1 \in C$.
  Then the following system is full:
  \begin{enumerate}
  \item $F_0 = C$ (i.e., a truth constant for each $c \in C$).
  \item $F_2 = \{\dotminus\}$.
  \item $F_n = \emptyset$ otherwise.
  \end{enumerate}
\end{cor}

\begin{cor}
  \label{cor:FullFin}
  The following system is full:
  \begin{enumerate}
  \item $F_1 = \{\lnot, \frac{x}{2}\}$.
  \item $F_2 = \{\dotminus\}$.
  \item $F_n = \emptyset$ otherwise.
  \end{enumerate}
\end{cor}

The system appearing in \fref{cor:Full}
can be viewed as the continuous analogue of the full
system of Boolean connectives $\{T,F,\to\}$ ($T$ and $F$ being truth constants),
while that of \fref{cor:FullFin} is reminiscent of $\{\lnot,\to\}$.
We will usually use the latter (i.e.,
$\{\lnot,\frac{x}{2},\dotminus\}$), which has the advantage of being
finite.
Note however that for this we need to introduce an
additional unary connective $\frac{x}{2}$
which has no counterpart in classical
discrete logic.

\begin{rmk}
  \label{rmk:NotFull}
  Unlike the discrete case, the family $\{\lnot,\lor,\land\}$ is not full, and this
  cannot be remedied by the addition of truth constants.
  Indeed, it can be verified by induction that every function
  $f\colon [0,1]^n \to [0,1]$ constructed from these connectives is
  $1$-Lipschitz in every argument.
\end{rmk}

This takes care of connectives: any full system would do.
We will usually prefer to work with countable systems of connectives,
so that countable signatures give countable languages.
When making general statements (e.g., the axioms for pseudo-metrics
and uniform continuity we give below) it is advisable to use a
minimal system of connectives, and we will usually use the one
from \fref{cor:FullFin} consisting of
$\{\lnot,\frac{x}{2},\dotminus\}$.
On the other hand, when spelling out actual theories, it may be
convenient (and legitimate) to admit additional continuous functions
from $[0,1]^n$ to $[0,1]$ as connectives.

As for quantifiers, the situation is much simpler: we contend that
the
transition from $\{T,F\}$ to $[0,1]$ imposes a single pair of
quantifiers, or rather, imposes a re-interpretation of the classical
quantifiers $\forall$ and $\exists$
(on this point we differ quite significantly from \cite{Chang-Keisler:ContinuousModelTheory}).
In order to see this, let us look for a construction of the discrete
quantifiers $\forall$ and $\exists$.

Let $M$ be a set, and $\cR_M(n)$ be the set of all $n$-ary relations
on $M$; we may view each $R \in \cR_M(n)$ as a property of $n$ free
variables $x_0,\ldots,x_{n-1}$.
For every $R \in \cR_M(n)$, let $j(R) \in \cR_M(n+1)$ be defined as the
same relation, with an additional dummy variable $x_n$.
Similarly, for every $R \in \cR_M(n+1)$ we have relations $\exists x_n\,R$
and $\forall x_n\, R$ in $\cR_M(n)$.
Then for every $R \in \cR_M(n)$ and $Q \in \cR_M(n+1)$ the following two
properties hold
(here ``$\models$'' means ``implies''):
\begin{align*}
  Q \models j(R) & \Longleftrightarrow \exists x_n\,Q \models R\\
  j(R) \models Q & \Longleftrightarrow R \models \forall x_n\,Q.
\end{align*}
These properties actually determine the relations $\exists x_n\,Q$ and
$\forall x_n\,Q$, and can therefore be used as the definition of the
semantics of the quantifiers.

Replacing $\{T,F\}$ with $[0,1]$, let $\cC_M(n)$ be the set of all
functions from $M^n$ to $[0,1]$.
We define $j\colon \cC_M(n) \to \cC_M(n+1)$ as above, and
$\inf_{x_n},\sup_{x_n} \colon \cC_M(n+1) \to \cC_M(n)$ in the obvious
manner.
Since we identify $T$ with $0$ and $F$ with $1$, the relation $\models$
should be replaced with $\geq$, and we observe that for every
$f \in \cC_M(n)$ and $g \in \cC_M(n+1)$:
\begin{align*}
  g \geq j(f) & \Longleftrightarrow \inf_{x_n} g \geq f\\
  j(f) \geq g & \Longleftrightarrow f \geq \sup_{x_n} g.
\end{align*}
Therefore, as in discrete logic, we will have two quantifiers, whose
semantics are defined by the properties above.
We will use the symbols $\inf$ and $\sup$, respectively, to denote the
quantifiers, as these best describe their semantics.
Make no mistake, though: these are not ``new'' quantifiers that we
have ``chosen'' for continuous logic, but rather the only possible
re-interpretation of the discrete quantifiers $\exists$ and $\forall$ in
continuous logic.
(\fref{rmk:AEntn} below will relate our quantifiers to Henson's
sense of approximate satisfaction of quantifiers, further justifying
our choice of quantifiers.)

Once we have connectives and quantifiers, we define the set of
continuous first order formulae in the usual manner.

\begin{dfn}
  A \emph{condition} is an expression of the form $\varphi = 0$ where $\varphi$ is a
  formula.

  A condition is \emph{sentential} if
  $\varphi$ is a sentence.

  A condition is \emph{universal} if it is of the form
  $\sup_{\bar x} \varphi = 0$ where $\varphi$ is quantifier-free
\end{dfn}

If $r$ is a dyadic number then $\varphi \leq r$ is an abbreviation for
the condition $\varphi \dotminus r = 0$, and similarly $\varphi \geq r$ for
$r \dotminus \varphi = 0$.
(Thus $\sup_{\bar x} \varphi \leq r$ and $\inf_{\bar x} \varphi \geq r$ abbreviate
universal conditions.)
With some abuse of notation we may use $\varphi \leq r$ for an arbitrary
$r \in [0,1]$ as an abbreviation for the \emph{set} of conditions
$\{\varphi \leq r'\colon r' \geq r \text{ dyadic}\}$.
We define $\varphi \geq r$ and $\varphi = r$ as abbreviations for sets of conditions
similarly.

\begin{ntn}
  \label{ntn:UnivCond}
  Given a formula $\varphi$
  we will use $\forall \bar x\, \varphi = 0$ as an alternative notation for
  $\sup_{\bar x} \varphi = 0$.
  While this may be viewed as a mere notational convention,
  the semantic contents of $\forall\bar x\, (\varphi = 0)$ is indeed
  equivalent to that of $(\sup_{\bar x} \varphi) = 0$
  (notice how the parentheses move, though).
  Similarly, we may write $\forall\bar x\, \varphi \leq r$ for
  $\sup_{\bar x} \varphi \leq r$ and $\forall\bar x\, \varphi \geq r$ for
  $\inf_{\bar x} \varphi \geq r$.
\end{ntn}

\section{Continuous structures}
\label{sec:struct}

In classical logic one usually has a distinguished binary predicate
symbol $=$, and the logic requires that this
symbol always be interpreted as actual equality.
The definition we gave for a non-metric continuous signature is the
analogue of a discrete signature \emph{without equality}.
The analogue of a discrete signature with equality is somewhat
trickier, since the symbol taking equality's place need no longer be
discrete.
Discrete equality always satisfies the equivalence relation axioms:
\begin{align*}
  \tag{ER}\label{ax:er}
  & \begin{aligned}
    & \forall x\,x=x\\
    & \forall xy\,x=y\to y=x\\
    & \forall xyz\,x=y \to (y=z\to x=z)\\
  \end{aligned}
\end{align*}
Still within the discrete framework, let us replace the symbol $=$
with the symbol $d$.
Recalling that $T = 0$, $F = 1$ we obtain the discrete metric:
\begin{gather*}
  d(a,b) =
  \begin{cases}
    0 & a = b\\
    1 & a \neq b
  \end{cases}
\end{gather*}
Let us now translate \fref{ax:er} to continuous logic, recalling that
$\dotminus$ is the analogue of implication:
\begin{align*}
  \tag{PM}\label{ax:pm}
  & \begin{aligned}
    & \sup_x\, d(x,x) = 0\\
    & \sup_{xy}\, d(x,y) \dotminus d(y,x) = 0\\
    & \sup_{xyz}\, (d(x,z) \dotminus d(y,z)) \dotminus d(x,y) = 0
  \end{aligned}\\
\end{align*}
Following \fref{ntn:UnivCond}, we can rewrite \fref{ax:pm}
equivalently as the axioms of a pseudo-metric, justifying the
use of the symbol $d$:
\begin{align}
  \tag{PM$'$}\label{ax:pmp}
  & \begin{aligned}
    & \forall x\, d(x,x) = 0\\
    & \forall xy\, d(x,y) = d(y,x)\\
    & \forall xyz\, d(x,y) \leq d(x,z)+d(z,y)
  \end{aligned}
\end{align}

By the very definition of equality it is also a
congruence relation for all the other symbols, which can be
axiomatised as:
\begin{align*}
  \tag{CR}\label{ax:cr}
  & \begin{aligned}
    & \forall\bar x\bar yzw\,\big(z=w\to f(\bar x,z,\bar y)=f(\bar x,w,\bar y)\big)\\
    & \forall\bar x\bar yzw\,\big(z=w\to(P(\bar x,z,\bar y)\to P(\bar x,w,\bar y))\big)
  \end{aligned}
\end{align*}
Translating \fref{ax:cr} to continuous logic as we did with
\fref{ax:er} above would yield axioms saying that every symbol is
$1$-Lipschitz with respect to $d$ in each of the variables.
While we could leave it like this, there is no harm in allowing other
moduli of uniform continuity.
\begin{dfn}
  \label{dfn:ContMod}
  \begin{enumerate}
  \item A \emph{continuity modulus} is a
    function $\delta\colon (0,\infty) \to (0,\infty)$
    (for our purposes a domain of $(0,1]$ would suffice).
  \item Let $(X_1,d_1)$, $(X_2,d_2)$ be metric spaces.
    We say that a mapping $f\colon X_1 \to X_2$ is \emph{uniformly continuous
      with respect} to a continuity
    modulus $\delta$ (or just that $f$ \emph{respects} $\delta$)
    if for all $\varepsilon > 0$ and for all $x, y \in X_1$:
    $d_1(x,y) < \delta(\varepsilon) \Longrightarrow d_2(f(x),f(y)) \leq \varepsilon$.
  \end{enumerate}
  (For a different approach to the definition of continuity moduli see
  \fref{apdx:InvContMod}.)
\end{dfn}
Thus, for each $n$-ary symbol $s$ and each $i < n$ we will fix
a continuity modulus
$\delta_{s,i}$, and the congruence relation property will be replaced with
the requirement that as a function of its $i$th argument,
$s$ should respect $\delta_{s,i}$.
As above this can be written in pure continuous logic or be translated
to a more readable form:
\begin{align}
  \tag{UC$_\cL$}\label{ax:uc}
  & \begin{aligned}
    & \sup_{x_{<i},y_{<n-i-1},z,w}\,
    \big( \delta_{f,i}(\varepsilon) \dotminus d(z,w) \big) \land
    \big( d(f(\bar x,z,\bar y),f(\bar x,w,\bar y))\dotminus\varepsilon \big)
     = 0\\
    & \sup_{x_{<i},y_{<n-i-1},z,w}\,
    \big(\delta_{P,i}(\varepsilon) \dotminus d(z,w) \big) \land
    \big(
    (P(\bar x,z,\bar y)\dotminus P(\bar x,w,\bar y))\dotminus\varepsilon
    \big) = 0
  \end{aligned}\\
  \tag{UC$'_\cL$}\label{ax:ucp}
  & \begin{aligned}
    & \forall x_{<i},y_{<n-i-1},z,w\,
    \Big(d(z,w) < \delta_{f,i}(\varepsilon) \to d(f(\bar x,z,\bar y),f(\bar x,w,\bar y))\leq \varepsilon\Big)\\
    & \forall x_{<i},y_{<n-i-1},z,w\,
    \Big(d(z,w) < \delta_{P,i}(\varepsilon) \to
    \big( P(\bar x,z,\bar y) \dotminus P(\bar x,w,\bar y) \big)\leq\varepsilon\Big)
  \end{aligned}
\end{align}
Here $x_{<i}$ denotes the tuple $x_0,\ldots,x_{i-1}$, and similarly for
$y_{<n-i-1}$, etc.

\begin{rmk}
  The axiom scheme \fref{ax:uc} can be reformulated to mention only
  constants in some dense set $C \subseteq [0,1]$ (say rational or
  dyadic numbers): simply, for every $\varepsilon > 0$, and every $r,q \in C$ such
  that $r > \varepsilon$ and $q < \delta_{s,i}(\varepsilon)$ (where $s$ is either $f$ or $P$):
  \begin{align*}
    \tag{UC$''_\cL$}\label{ax:ucpp}
    & \begin{aligned}
      & \sup_{\bar x,\bar y,z,w}\,
      \big(q \dotminus d(z,w) \big) \land
      \big(
      d(f(\bar x,z,\bar y),f(\bar x,w,\bar y))\dotminus r
      \big)
      = 0\\
      & \sup_{\bar x,\bar y,z,w}\,
      \big(q \dotminus d(z,w) \big) \land
      \big(
      (P(\bar x,z,\bar y)\dotminus P(\bar x,w,\bar y))\dotminus r
      \big) = 0
    \end{aligned}
  \end{align*}
\end{rmk}

This leads to the following definition:
\begin{dfn}
  A \emph{(metric) continuous signature} is a non-metric continuous
  signature along with the following additional data:
  \begin{enumerate}
  \item One binary predicate symbol, denoted $d$, is specified as the
    distinguished \emph{distance symbol}.
  \item For each $n$-ary symbol $s$, and for each $i < n$
    a continuity modulus $\delta_{s,i}$, called
    the \emph{uniform continuity modulus} of $s$ with respect to the
    $i$th argument.
  \end{enumerate}

  If we work with a multi-sorted signature then each sort $S$ has
  its own distinguished distance symbol $d_S$.
\end{dfn}

\begin{dfn}
  Let $\cL$ be a continuous signature.
  A \emph{(continuous) $\cL$-pre-structure} is a set $M$
  equipped, for every $n$-ary function symbol $f \in \cL$, with a
  mapping $f^M\colon M^n \to M$, and for every $n$-ary relation symbol
  $P \in \cL$, with a mapping $P^M\colon M^n \to [0,1]$, such that the
  pseudo-metric
  and uniform continuity axioms \fref{ax:pm}, \fref{ax:uc}
  (or equivalently \fref{ax:pmp}, \fref{ax:ucp}) hold.

  An \emph{$\cL$-structure} is a pre-structure $M$ in which $d^M$ is a
  complete metric (i.e., $d^M(a,b) = 0 \Longrightarrow a = b$ and every Cauchy
  sequence converges).
\end{dfn}

The requirement that $d^M$ be a metric corresponds to the
requirement that $=^M$ be equality.
Completeness, on the other hand, has no analogue in discrete
structures, since every discrete metric is trivially complete; still,
it turns out to be the right thing to require.

As in classical logic, by writing a term $\tau$ as $\tau(\bar x)$ we mean
that all variables occurring in $\tau$ appear in $\bar x$.
Similarly, for a formula $\varphi$ the notation $\varphi(\bar x)$ means that the
tuple $\bar x$ contains all free variables of $\varphi$.

\begin{dfn}
  \label{dfn:TermInterp}
  Let $\tau(x_{<n})$ be a term, $M$ a $\cL$-pre-structure.
  The \emph{interpretation of $\tau(\bar x)$ in $M$} is a function
  $\tau^M\colon M^n \to M$ defined inductively as follows:
  \begin{itemize}
  \item If $\tau = x_i$ then $\tau^M(\bar a) = a_i$.
  \item If $\tau = f(\sigma_0,\ldots,\sigma_{m-1})$ then
    $\tau^M(\bar a) = f^M\big( \sigma_0^M(\bar a),\ldots,\sigma_{m-1}^M(\bar a) \big)$.
  \end{itemize}
\end{dfn}

\begin{dfn}
  \label{dfn:FormInterp}
  Let $\varphi(x_{<n})$ be a term, $M$ a $\cL$-pre-structure.
  The \emph{interpretation of $\varphi(\bar x)$ in $M$} is a function
  $\varphi^M\colon M^n \to [0,1]$ defined inductively as follows:
  \begin{itemize}
  \item If $\varphi = P(\tau_0,\ldots,\tau_{m-1})$ is atomic then
    $\varphi^M(\bar a) = P^M\big( \tau_0^M(\bar a),\ldots,\tau_{m-1}^M(\bar a) \big)$.
  \item If $\varphi = \lambda(\psi_0,\ldots,\psi_{m-1})$ where $\lambda$ is a continuous connective
    then
    $\varphi^M(\bar a) = \lambda\big( \psi_0^M(\bar a),\ldots,\psi_{m-1}^M(\bar a) \big)$.
  \item If $\varphi = \inf_y \psi(y,\bar x)$
    $\varphi^M(\bar a) = \inf_{b\in M} \psi^M(b,\bar a)$, and similarly for
    $\sup$.
  \end{itemize}
\end{dfn}

\begin{prp}
  \label{prp:UnifContLanguage}
  Let $M$ be an $\cL$-pre-structure, $\tau(x_{<n})$ a term,
  $\varphi(x_{<n})$ a formula.
  Then the mappings $\tau^M\colon M^n \to M$ and $\varphi^M\colon M^n \to [0,1]$ are
  uniformly continuous in each of their
  arguments.
  Moreover, $\tau^M$ and $\varphi^M$ respect uniform continuity moduli which depend
  only on $\tau$ and $\varphi$ but not on $M$.
\end{prp}
\begin{proof}
  In the case of terms, this is just an inductive argument using the
  fact that a composition of
  uniformly continuous mappings is uniformly continuous.
  In the case of formulae one needs two more facts.
  First, all connectives are uniformly continuous as continuous
  mappings from a compact space.
  Second, if $\varphi(\bar x) = \inf_y \psi(y,\bar x)$ then any uniform
  continuity modulus $\psi(y,\bar x)$ respects with respect to $x_i$ is
  also respected by $\varphi$.

  As the uniform continuity proof above does not depend on $M$ in any
  way, uniform continuity moduli for terms and formulae can be
  extracted from the inductive argument.
\end{proof}

\begin{dfn}
  \label{dfn:CondSat}
  \begin{enumerate}
  \item Let $s(\bar x)$ be a condition $\varphi(\bar x) = 0$,
    $M$ an $\cL$-(pre-)structure and $\bar a \in M$.
    We say that $s$ is satisfied by $\bar a$ in $M$, in symbols
    $M \models s(\bar a)$, or even $\bar a \models s$ (in case the ambient
    structure $M$ is clear from the context) if
    $\varphi^M(\bar a) = 0$.
  \item A set of conditions $\Sigma(\bar x)$ is \emph{satisfied} by a tuple
    $\bar a \in M$ (again denoted $M \models \Sigma(\bar a)$ or $\bar a \models \Sigma$)
    if all conditions in $\Sigma$ are satisfied by $\bar a$ in $M$.
    This makes sense just as well in case $\Sigma$ involves
    infinitely many free variables.
  \item
    Earlier on we defined $\varphi(\bar x) \leq r$, $\varphi(\bar x) \geq r$, etc., as
    abbreviations for conditions or sets thereof.
    Notice that $\varphi^M(\bar a) \leq r$ if and only if the corresponding
    (set of) condition(s) holds for $\bar a$, and similarly for the
    other abbreviations.
    We may therefore extend the definition above without ambiguity to
    satisfaction of such abbreviations.
  \item If $\Sigma(\bar x)$ is a set of conditions and $s(\bar x)$
    a condition, we say that $s$ is a \emph{logical consequence} of $\Sigma$,
    or that $\Sigma$ \emph{implies} $s$, in symbols $\Sigma \models s$, if for every
    $M$ and $\bar a \in M$: $M \models \Sigma(\bar a) \Longrightarrow M \models s(\bar a)$.
  \item A set of conditions $\Sigma$ is \emph{satisfiable}
    if there is a structure
    $M$ and $\bar a \in M$ such that $M \models \Sigma(\bar a)$.
    It is \emph{finitely satisfiable} if every finite
    $\Sigma_0 \subseteq \Sigma$ is satisfiable.
    We may further say that $\Sigma$ is \emph{approximately finitely
      satisfiable} if for every finite subset $\Sigma_0 \subseteq \Sigma$, which we may
    assume to be of the form $\{\varphi_i = 0\colon i < n\}$, and for every
    $\varepsilon >0$,
    the set of conditions $\{\varphi_i \leq \varepsilon\colon i < n\}$ is satisfiable.
  \end{enumerate}
\end{dfn}

\begin{dfn}
  A \emph{morphism} of $\cL$-pre-structures is a mapping of the
  underlying sets which preserves the interpretations of the symbols.
  It is \emph{elementary} if it preserves the truth values of formulae
  as well.
\end{dfn}

\begin{prp}
  \label{prp:cmplt}
  Let $M$ be an $\cL$-pre-structure.
  Let $M_0 = M/\{d^M(x,y)=0\}$, and let $d_0$ denote the metric
  induced by $d^M$ on $M_0$.
  Let $(\hat M_0,\hat d_0)$ be the completion of the metric space
  $(M_0,d_0)$ (which is for all intents and purposes unique).

  Then there exists a unique way to define an $\cL$-structure
  $\hat M$ on the set $\hat M_0$ such that $d^{\hat M} = \hat d_0$
  and the natural
  mapping $M \to \hat M$ is a morphism.
  We call $\hat M$ the
  $\cL$-structure \emph{associated to} $M$.

  Moreover:
  \begin{enumerate}
  \item If $N$ is any other $\cL$-structure, then any morphism $M \to N$
    factors uniquely through $\hat M$.
  \item The mapping $M \to \hat M$ is elementary.
  \end{enumerate}

  Another way of saying this is that the functor $M \mapsto \hat M$
  is the left adjoint of the forgetful functor from the category of
  $\cL$-structures to that of $\cL$-pre-structures, and that it
  sends elementary morphisms to elementary morphisms.
\end{prp}
\begin{proof}
  Straightforward using standard facts about metrics, pseudo-metrics
  and completions.
\end{proof}

We say that two formulae are \emph{equivalent}, denoted $\varphi \equiv \psi$ if they
define the same functions on every $\cL$-structure (equivalently: on
every $\cL$-pre-structure).
For example, let $\varphi[t/x]$ denote the free substitution of $t$ for $x$
in $\varphi$.
Then if $y$ does not appear in $\varphi$, then
$\sup_x \varphi \equiv \sup_y \varphi[y/x]$ (this is bound substitution of $y$ for $x$
in $\sup_x \varphi$).
Similarly, provided that $x$ is not free in $\varphi$ we have
$\varphi\land\sup_x \psi \equiv \sup_x \varphi\land\psi$,
$\varphi \dotminus \sup_x \psi \equiv \sup_x (\varphi \dotminus \psi)$, etc.

Using these and similar observations, it is easy to verify that all
formulae written using the full system of connectives
$\{\lnot,\half[x],\dotminus\}$ have equivalent prenex forms.
In other words, for every such formula $\varphi$ there is an equivalent formula
of the form $\psi = \sup_x \inf_y \sup_z \ldots \varphi$, where $\varphi$ is
quantifier-free.
The same would hold with any other system of connectives which are
monotone in each of their arguments.

\begin{rmk}
  \label{rmk:AEntn}
  We can extend \fref{ntn:UnivCond} to all conditions in prenex
  form, and thereby to all conditions.
  Consider a condition $\psi \leq r$ (a condition of the basic form
  $\psi = 0$ is equivalent to $\psi \leq 0$).
  Write $\psi$ in prenex form, so the condition becomes:
  \begin{gather*}
    \sup_x \left( \inf_y \left( \sup_z \ldots \varphi(x,y,z,\ldots)\right)\right)
    \leq r
  \end{gather*}
  A reader familiar with Henson's logic \cite{Henson-Iovino:Ultraproducts}
  will not find it difficult to
  verify, by induction on the number of quantifiers,
  that this is equivalent to the \emph{approximate satisfaction} of:
  \begin{gather*}
    \forall x \Big( \exists y \Big( \forall z\ldots\,\varphi(x,y,z,\ldots) \leq r \Big)\Big)
  \end{gather*}
  (Notice how the parentheses move, though.)
\end{rmk}

We view this as additional evidence to the analogy between
the continuous quantifiers $\inf$ and $\sup$ and the Boolean
quantifiers $\exists$ and $\forall$.

\begin{rmk}
  Unlike the situation in Henson's logic, there are no bounds on the
  quantifiers as everything in our logic is already assumed to be
  bounded.
  For a fuller statement of equivalence between satisfaction in
  continuous logic and approximate satisfaction in positive bounded
  logic, see the section on unbounded structure in
  \cite{BenYaacov:Perturbations}.
  In particular we show there that under appropriate modifications
  necessitated by the fact that Henson's logic considers unbounded
  structures, it has the same power of expression as continuous logic.
\end{rmk}

\begin{dfn}
  Let $M$ be an $\cL$-structure.
  A \emph{formula with parameters in $M$} is something of the form
  $\varphi(\bar x,\bar b)$, where $\varphi(\bar x,\bar y)$ is a formula in the
  tuples of variables $\bar x$ and $\bar y$, and
  $\bar b \in M$.
  Such a formula can also be viewed as an $\cL(M)$-formula, where
  $\cL(M)$ is obtained from $\cL$ by adding constant symbols for the
  elements of $M$, in which case it may be denoted by $\varphi(\bar x)$
  (i.e., the parameters may be ``hidden'').
\end{dfn}

\begin{dfn}[Ultraproducts]
  Let $\{M_i\colon i \in I\}$ be $\cL$-structures (or even pre-structures), and
  $\cU$ an ultrafilter on $I$.
  Let $N_0 = \prod_i M_i$, and interpret the function and predicate symbols on
  it as follows:
  \begin{gather*}
    f^{N_0}((a_i),(b_i),\ldots) = (f^{M_i}(a_i,b_i,\ldots))\\
    P^{N_0}((a_i),(b_i),\ldots) = \lim_\cU P^{M_i}(a_i,b_i,\ldots)
  \end{gather*}
  Recall that a sequence in a compact set has a unique limit modulo an
  ultrafilter: for any open set $U \subseteq [0,1]$, we have:
  \begin{gather*}
    P^{N_0}((a_i),(b_i),\ldots) \in U \Longleftrightarrow \{i \in I\colon P^{M_i}(a_i,b_i,\ldots) \in U\} \in \cU
  \end{gather*}
  It is immediate to verify that $N_0$ satisfies \fref{ax:pm} and
  \fref{ax:uc}, so $N_0$ is an $\cL$-pre-structure.
  Finally, define $\prod_i M_i/\cU = N = \hat N_0$, and call it the
  \emph{ultraproduct} of $\{M_i\colon i \in I\}$ modulo $\cU$.
\end{dfn}

\begin{thm}[\L o\'s's Theorem for continuous logic]
  Let $N = \prod_i M_i/\cU$ as above.
  For every tuple $(a_i) \in \prod M_i$ let $[a_i]$ be its image in $N$.
  Then for every formula $\varphi(\bar x)$ we have:
  \begin{gather*}
    \varphi^N([a_i],[b_i],\ldots) = \lim_\cU \varphi^{M_i}(a_i,b_i,\ldots)
  \end{gather*}
\end{thm}
\begin{proof}
  By induction on the complexity of $\varphi$.
  See also \cite[Chapter~V]{Chang-Keisler:ContinuousModelTheory}.
\end{proof}

\begin{cor}[Compactness Theorem for continuous first order logic]
  \label{cor:Cpt}
  Let $\Sigma$ be a family of conditions (possibly with free variables).
  Then $\Sigma$ is satisfiable in an $\cL$-structure if and only if it is
  finitely so, and furthermore if and only if it is approximately
  finitely so (see \fref{dfn:CondSat}).
\end{cor}
\begin{proof}
  The proof is essentially the same as in discrete logic.
  Replacing free variables with new constant symbols we may assume all
  conditions are sentential.
  Enumerate $\Sigma = \{\varphi_i = 0\colon i < \lambda\}$, and let
  $I = \{(w,\varepsilon)\colon w \subseteq \lambda \text{ is finite  and } \varepsilon > 0\}$.
  For every $(w,\varepsilon) \in I$ choose $M_{\omega,\varepsilon}$ in which
  the conditions $\varphi_i \leq \varepsilon$ hold for $i \in w$.
  For $(w,\varepsilon) \in I$, let
  $J_{w,\varepsilon} = \{(w',\varepsilon') \in I\colon w' \supseteq w, \varepsilon' \leq \varepsilon\}$.
  Then the collection $\sU_0 = \{J_{w,\varepsilon}\colon (w,\varepsilon) \in I\}$ generates a proper
  filter on $I$ which may be extended to an ultrafilter $\sU$.
  Let $M = \prod M_{w,\varepsilon} / \sU$.
  By \L o\'s's Theorem we have $M \models \varphi_i \leq \varepsilon$ for every $i < \lambda$ and $\varepsilon>0$,
  so in fact $M \models \varphi_i = 0$.
  Thus $M \models \Sigma$.
\end{proof}

\begin{fct}[Tarski-Vaught Test]
  \label{fct:TVTest}
  Let $M$ be a structure, $A \subseteq M$ a closed subset.
  Then the following are equivalent:
  \begin{enumerate}
  \item The set $A$ is (the domain of) an elementary substructure of
    $M$: $A \preceq M$.
  \item For every formula $\varphi(y,\bar x)$ and every $\bar a \in A$:
    $$\inf \{\varphi(b,\bar a)^M\colon b \in M\} = \inf \{\varphi(b,\bar a)^M\colon b \in A\}.$$
  \end{enumerate}
\end{fct}
\begin{proof}
  One direction is by definition.
  For the other, we first verify that $A$ is a substructure of $M$,
  i.e., closed under the function symbols.
  Indeed, in order to show that $\bar a \in A \Longrightarrow f(\bar a) \in A$ we use
  the assumption for the formula $d(y,f(\bar a))$ and the fact that
  $A$ is complete.
  We then proceed to show that $\varphi(\bar a)^A = \varphi(\bar a)^M$ for all
  $\bar a \in A$ and formula $\varphi$ by induction on $\varphi$, as in the first
  order case.
\end{proof}

When measuring the size of a structure we will use its density
character (as a metric space), denoted $\|M\|$, rather than its
cardinality.

We leave the following results as an exercise to the reader:
\begin{fct}[Upward L\"owenheim-Skolem]
  Let $M$ be a non-compact structure (as a metric space).
  Then for every cardinal $\kappa$ there is an elementary extension $N \succeq M$
  such that $\|N\| \geq \kappa$.
\end{fct}

\begin{fct}[Downward L\"owenheim-Skolem]
  \label{fct:DLS}
  Let $M$ be a structure, $A \subseteq M$ a subset.
  Then there exists an elementary substructure $N \preceq M$ such that
  $A \subseteq N$ and $\|N\| \leq |A| + |\cL|$.
\end{fct}

\begin{fct}[Elementary chain]
  Let $\alpha$ be an ordinal and $(M_i\colon i < \alpha)$ an increasing chain of
  structures such that $i < j < \alpha \Longrightarrow M_i \preceq M_j$.
  Let $M = \bigcup_i M_i$.
  Then $M_i \preceq M$ for all $i$.
\end{fct}

\section{Types and definable predicates}
\label{sec:TypesDefPred}

We fix a continuous signature $\cL$, as well as a full system of
connectives (which might as well be $\{\lnot,\frac{x}{2},\dotminus\}$).

\subsection{Spaces of complete types}

Recall that a condition (or abbreviation thereof)
is something of the form $\varphi = 0$, $\varphi \leq r$ or $\varphi \geq r$
where $\varphi$ is a formula and $r \in [0,1]$ is dyadic.
\begin{dfn}
  \begin{enumerate}
  \item Let $M$ be a structure and $\bar a \in M^n$.
    We define the \emph{type of $\bar a$ in $M$}, denoted
    $\tp^M(\bar a)$ (or just $\tp(\bar a)$ when there is no ambiguity
    about the ambient structure), as the set of all conditions in
    $\bar x = x_{<n}$ satisfied in $M$ by $\bar a$.
  \item Observe that equality of types $\tp^M(\bar a) = \tp^N(\bar b)$
    is equivalent to the elementary equivalence
    $(M,\bar a) \equiv (N,\bar b)$ (where the tuples are named by new
    constant symbols).
    If the ambient structures
    are clear for the context we may shorten this to
    $\bar a \equiv \bar b$.
  \item A \emph{complete $n$-type}, or just an \emph{$n$-type}, is a
    maximal satisfiable set of conditions in the free variables $x_{<n}$.
    The space of all $n$-types is denoted $\tS_n$.
  \item If $p$ is an $n$-type and $\bar a \in M^n$ is such that
    $M \models p(\bar a)$, we say that $\bar a$ \emph{realises}
    $p$ (in $M$).
  \end{enumerate}
\end{dfn}

We start with a few easy observations whose proof we leave to the
reader.
\begin{lem}
  \label{lem:Types}
  \begin{enumerate}
  \item If $\bar a$ is an $n$-tuple (in some structure) then
    $\tp(\bar a)$ is a (complete) $n$-type.

    Conversely, every $n$-type can be obtained as the type of an
    $n$-tuple.
  \item Let $p \in \tS_n$ and $\varphi(x_{<n})$ be a formula.
    Then for every realisation $\bar a \models p$, the value of
    $\varphi(\bar a)$ depends only on $p$.
    It will be denoted $\varphi^p$ (the value of $\varphi$ according to $p$).

    Conversely, the mapping $\varphi \mapsto \varphi^p$, where $\varphi$ varies over all
    formulae in the variables $x_{<n}$, determines $p$.
 \end{enumerate}
\end{lem}

Every formula $\varphi(x_{<n})$ defines a mapping
$p \mapsto \varphi^p\colon \tS_n \to [0,1]$.
With some abuse of notation this function will also be denoted
by $\varphi$.
This is legitimate, since it is clear that
two formulae $\varphi(\bar x)$ and $\psi(\bar x)$
are equivalent if and only if the functions
$\varphi,\psi\colon \tS_n \to [0,1]$ are equal.

We equip $\tS_n$ with the minimal topology in which all the
functions of this form are continuous, which
is sometimes called the \emph{logic topology}.

\begin{lem}
  \label{lem:CompactTypeSpace}
  With the topology given above,
  $\tS_n$ is a compact and Hausdorff space.
\end{lem}
\begin{proof}
  Let $\cL(n)$ be the family of all formulae in the variables
  $x_{<n}$, and consider the mapping $\tS_n \to [0,1]^{\cL(n)}$ given by
  $p \mapsto (\varphi^p\colon \varphi \in \cL(n))$.
  As observed in \fref{lem:Types}
  this mapping is injective, and the topology on
  $\tS_n$ is the topology induced on its image from the product
  topology on $[0,1]^{\cL(n)}$, which is Hausdorff.
  It follows from the Compactness Theorem (\fref{cor:Cpt}) that
  the image is closed in $[0,1]^{\cL(n)}$, so it is compact.
\end{proof}

Since we are only interested in formulae up to equivalence, it is
legitimate to identify a formula with such a mapping.
We know that in discrete first order logic the equivalence classes of
formulae are in bijection with clopen sets of types, i.e., with
continuous mappings from the type space to $\{T,F\}$.
Here we cannot claim as much, and it only holds up to uniform
approximations:
\begin{prp}
  \label{prp:TypeFuncApprox}
  A function $f\colon \tS_n \to [0,1]$ is continuous if and only if it can be
  uniformly approximated by formulae.
\end{prp}
\begin{proof}
  For right to left, we know that every formula defines a continuous
  mapping on the type space, and a uniform limit of continuous
  mappings is continuous.

  Left to right is a consequence of \fref{cor:SWCnct} and the
  fact that formulae separate types.
\end{proof}

Given a formula $\varphi(x_{<n})$ and $r \in [0,1]$, we define
$[\varphi < r]^{\tS_n} = \{p \in \tS_n\colon \varphi^p < r\}$.
We may omit $\tS_n$ from the notation when it may not cause ambiguity.
We could define $[\varphi > r]$ similarly, but as it is equal to
$[\lnot\varphi < 1-r]$ this would not introduce any new sets.
All sets of this form are clearly open in $\tS_n$.
Similarly, we define sets of the form $[\varphi \leq r]$, $[\varphi \geq r]$, which are
closed.
(Since $\varphi \leq r$ is a condition we can also characterise $[\varphi \leq r]$
as the set $\{p\colon \text{``$\varphi \leq r$''} \in p\}$).

\begin{lem}
  The family of sets of the form $[\varphi < r]$ forms a basis of open sets
  for the topology on $\tS_n$.
  Equivalently, the family of sets of the form $[\varphi \leq r]$ forms a basis
  of closed sets.

  Moreover, if $U$ is a neighbourhood of $p$, we can always find a
  formula $\varphi(x_{<n})$ such that
  $p \in [\varphi = 0] \subseteq [\varphi < 1/2] \subseteq U$.
\end{lem}
\begin{proof}
  We prove the moreover part, which clearly implies the rest.
  Assume that $p \in U \subseteq \tS_n$ and $U$ is open.
  By Urysohn's Lemma there is a continuous function $f\colon \tS_n \to [0,1]$
  such that $f(p) = 0$ and $f\rest_{U^c} = 1$.
  We can then find a formula $\varphi_0$ such that $|f - \varphi_0| \leq 1/4$.
  Then the formula $\varphi = \varphi_0 \dotminus 1/4$ would do.
\end{proof}

\subsection{Definable predicates}

The discussion above, which is semantic in nature, should convince the
reader that uniform limits of formulae are interesting objects, which
we would like to call \emph{definable predicates}.
But, as formulae are first defined syntactically and only later
interpreted as truth value mappings from structures or from type
spaces, it will be more convenient later on to first define definable
predicates syntactically.
Since uniform convergence of the truth values is a semantic notion it
cannot be brought into consideration on the syntactic level, so we use
instead a trick we call \emph{forced convergence}.
This will be particularly beneficial later on when we need to consider
sequences of formulae which sometimes (i.e., in some structures, or
with some parameters) converge, and sometimes do not.

Forced conversion is first of all an operation on sequences in
$[0,1]$, always yielding a number in $[0,1]$.
The forced limit coincides with the limit if the sequence converges
fast enough.
More precisely, if the sequence converges fast enough,
we take its limit; otherwise, we find the closest sequence which does
converge fast enough, and take \emph{its} limit.
Formally:
\begin{dfn}
  Let $(a_n\colon n < \omega)$ be a sequence in $[0,1]$.
  We define a
  modified sequence $(a_{\flim,n}\colon n < \omega)$ by induction:
  \begin{align*}
    a_{\flim,0} & = a_0 \\
    a_{\flim,n+1} & =
    \begin{cases}
      a_{\flim,n} + 2^{-n-1} & a_{\flim,n} + 2^{-n-1} \leq a_{n+1} \\
      a_{n+1} &  a_{\flim,n} - 2^{-n-1} \leq a_{n+1} \leq a_{\flim,n} + 2^{-n-1} \\
      a_{\flim,n} - 2^{-n-1} & a_{\flim,n} - 2^{-n-1} \geq a_{n+1}
    \end{cases}
  \end{align*}
  The sequence $(a_{\flim,n}\colon n < \omega)$ is always a Cauchy sequence,
  satisfying $n \leq m < \omega \Longrightarrow |a_{\flim,n} - a_{\flim,m}| \leq 2^{-n}$.

  We define the \emph{forced limit} of the original sequence
  $(a_n\colon n < \omega)$ as:
  $$\flim_{n\to\infty} a_i \eqdef \lim_{n\to\infty} a_{\flim,n}.$$
\end{dfn}

\begin{lem}
  \label{lem:FLim}
  The function $\flim\colon [0,1]^\omega \to [0,1]$ is continuous, and if
  $(a_n\colon n < \omega)$ is a sequence such that $|a_n - a_{n+1}| \leq 2^{-n}$
  for all $n$ then $\flim a_n = \lim a_n$.

  In addition, if $a_n \to b \in [0,1]$ fast enough so that
  $|a_n - b| \leq 2^{-n}$ for all $n$, then $\flim a_n = b$.
\end{lem}
\begin{proof}
  Continuity follows from the fact that if $(a_n)$ and $(b_n)$ are two
  sequences, and for some $m$ we have
  $|a_n - b_n| < 2^{-m}$ for all $n \leq m$, then we can
  show by induction that
  $|a_{\flim,n} - b_{\flim,n}| < 2^{-m}$ for all $n \leq m$, whereby
  $|\flim a_n - \flim b_n| < 3\cdot2^{-m}$.

  For the second condition, one again shows by induction that
  $|a_{\flim,n} - b| < 2^{-n}$ whence the conclusion.
\end{proof}

We may therefore think of $\flim$ as an \emph{infinitary} continuous
connective.

\begin{dfn}
  A \emph{definable predicate} is a forced limit of a sequence of
  formulae, i.e., an (infinite) expression of the form
  $\flim_{n\to\infty} \varphi_n$.

  We say that a variable $x$ is free in $\flim \varphi_n$ if it is free in
  any of the $\varphi_n$.
\end{dfn}

Note that a definable predicate may have infinitely (yet countably)
many free variables.
In practise, we will mostly consider forced limits of formulae with a
fixed (finite) tuple of free variables, but possibly with
parameters, so the limit might involve infinitely many
parameters.
We may write such a definable predicate as
$\psi(\bar x,B) = \flim \varphi_n(\bar x,\bar b_n)$, and say that
$\psi(\bar x,B)$ is obtained from
$\psi(\bar x,Y)$ ($= \flim \varphi_n(\bar x,\bar y_n)$)
by substituting the infinite
tuple of parameters $B = \bigcup \bar b_n$ in place of
the parameter variables $Y = \bigcup \bar y_n$.
(Another way to think about this is to view parameter variables as
constant symbols for which we do not have yet any interpretation in
mind: then indeed all the free variables are in $\bar x$.)
Later on in \fref{sec:Img} we will construct canonical parameters
for such instances of $\psi(\bar x,Y)$.
We shall see then that dealing with infinitely many parameters (or
parameter variables) is not too difficult as long as the tuple
$\bar x$ of actually free variable is finite.
So from now on, we will only consider definable predicates in
finitely many variables.

The semantic interpretation of definable predicates is as expected:
\begin{dfn}
  Let $\psi(\bar x) = \flim \varphi_n(\bar x)$ be a definable predicate in
  $\bar x = x_{<m}$.
  Then for every structure $M$ and $\bar a \in M^m$, we define
  $\psi^M(\bar a) = \flim \varphi_n^M(\bar a)$.
  Similarly, for $p \in \tS_m$ we define
  $\psi^p = \flim \varphi_n^p$.
\end{dfn}
As with formulae, if $\bar a \in M$ realises $p \in \tS_m$, then
$\psi^M(\bar a) = \psi^p$.
Therefore, two definable predicates $\psi(\bar x)$ and $\chi(\bar x)$
are equivalent (i.e., $\psi^M = \chi^M$ for all $M$)
if and only if they are equal as functions of types.

In the same manner we tend to identify two equivalent formulae, we
will tend to identify two equivalent definable predicates.
Thus, rather than viewing definable as infinite syntactic objects we
will rather view them as \emph{semantic} objects,
i.e., as continuous functions on the corresponding type space.
If one insists on making a terminological distinction
then the syntactic notion of a forced limit of
formulae could be called a \emph{limit formula}.

\begin{prp}
  The continuous functions $\tS_m \to [0,1]$ are precisely those given
  by definable predicates in $m$ free variables.
\end{prp}
\begin{proof}
  Let $\psi(x_{<m}) = \flim \varphi_n$ be a definable predicate.
  Then the mapping $p \mapsto (\varphi_n^p\colon n < \omega)$ is continuous, and composed
  with the continuous mapping $\flim\colon [0,1]^\omega \to [0,1]$ we
  get $\psi\colon \tS_m \to [0,1]$, which is therefore continuous.

  Conversely, let $f\colon \tS_m \to [0,1]$ be continuous.
  Then by \fref{prp:TypeFuncApprox} $f$
  can be uniformly approximated by formulae: for every $n$ we
  can find $\varphi_n(x_{<m})$ such that $|f - \varphi_n| \leq 2^{-n}$.
  Then $f = \lim \varphi_n = \flim \varphi_n$.
\end{proof}

Viewed as continuous functions of types, it is clear that the family
of definable predicates is closed under all continuous combinations,
finitary or not (so in particular under the infinitary
continuous connective $\flim$), as well as under uniform limits.
Finally, let $\psi(\bar x,y)$ be a continuous predicate, and find
formulae $\varphi_n(\bar x,y)$ that converge to $\psi$ uniformly on $\tS_{m+1}$.
Then the sequence $\sup_y \varphi_n$ converges uniformly on $\tS_m$ to a
definable predicate which will be denoted $\sup_y \psi$ (this is since
we always have $|\sup_y \varphi_n - \sup_y \varphi_k| \leq |\varphi_n - \varphi_k|$).
Furthermore, if follows from the uniform convergence $\varphi_n \to \psi$ that
for every structure $M$ and $\bar a \in M$ we have
$(\sup_y \psi)^M(\bar a) = \sup \{\psi^M(\bar a,b)\colon b \in M\}$.
The same holds with $\inf$, and we conclude that the family of
definable predicates is closed under continuous quantification.
Also, since all formulae are uniformly continuous with respect to the
metric $d$, so are their uniform limits.
This means, for example, that every definable predicate can be added
to the language as a new actual predicate symbol.

We conclude with a nice consequence of the forced limit construction,
which might not have been obvious if we had merely defined definable
predicates as uniform limits of formulae.
\begin{lem}
  \label{lem:DefPredInStr}
  Let $M$ be a structure, and let $(\varphi_n(\bar x)\colon n < \omega)$ be a sequence
  of formulae, or even definable predicates, such that
  the sequence of functions $(\varphi_n^M\colon n < \omega)$ converges uniformly to some
  $\xi\colon M^m \to [0,1]$ (but need not converge at all for
  any other structure instead of $M$).

  Then there is a definable predicate $\psi(\bar x)$ such that
  $\psi^M = \xi = \lim \varphi_n^M$ (i.e., $\xi$ is definable in $M$).
\end{lem}
\begin{proof}
  Up to passing to a sub-sequence, we may assume that
  $|\varphi_n^M - \xi| \leq 2^{-n}$, so $\xi = \flim \varphi_n^M = (\flim \varphi_n)^M$.
  would do.
\end{proof}

\subsection{Partial types}

\begin{dfn}
  A \emph{partial type} is a set of conditions, usually in a
  finite tuple of variables.
  (Thus every complete type is in particular a partial type.)
\end{dfn}

For a partial type $p(x_{<n})$ we define:
$$[p]^{\tS_n} = \bigcap_{s \in p} [s]^{\tS_n} = \{q \in \tS_n\colon p \subseteq q\}.$$
Since the sets of the form $[s]$ (where $s$ is a condition in
$x_{<n}$) form a basis of closed sets for the topology on $\tS_n$,
the sets of the form $[p]$ (where $p(x_{<n})$ is a partial type)
are precisely the closed subsets of $\tS_n$.

For every $n,m < \omega$ we have a natural restriction mapping
$\pi\colon \tS_{n+m} \to \tS_n$.
This mapping is continuous, and therefore closed (as is every
continuous mapping between compact Hausdorff spaces).
Let $p(x_{<n},y_{<m})$ be a partial type, defining a closed subset
$[p] \subseteq \tS_{n+m}$.
Then $\pi([p]) \subseteq \tS_n$ is closed as well, and therefore of the form
$[q(x_{<n})]$.

\begin{dfn}
  Let $p(\bar x,\bar y)$ be a partial type.
  We define $\exists\bar y\, p(\bar x,\bar y)$ to be any partial type
  $q(\bar x)$ (say the maximal one) satisfying
  $[q]^{\tS_n} = \pi([p]^{\tS_{n+m}})$.
\end{dfn}

By the previous argument we have:
\begin{fct}
  For every partial type $p(\bar x,\bar y)$, a partial type
  $\exists\bar y\,p(\bar x,\bar y)$ exists.
  Moreover, if $M$ is structure and $\bar a \in M$ then
  $M \models \exists\bar y\,p(\bar a,\bar y)$ if and only if
  there is $N \succeq M$ and $\bar b \in N$ such that
  $N \models p(\bar a,\bar b)$.
  In case $M$ is $\omega$-saturated
  (i.e., if every $1$-type over a finite tuple in $M$ is realised in
  $M$) then such $\bar b$ exists in $M$.
\end{fct}

\section{Theories}
\label{sec:Theories}

\begin{dfn}
  A \emph{theory} is a set of sentential conditions,
  i.e., things of the form $\varphi = 0$, where $\varphi$ is a sentence.
\end{dfn}
Thus in some sense a theory is an ``ideal''.
Since every condition of the form $\varphi \leq r$, $\varphi \geq r$ or $\varphi = r$ is
logically
equivalent to one of the form $\varphi' = 0$ (if $r$ is dyadic) or to a set
of such conditions (for any $r \in [0,1]$), we may allow ourselves
conditions of this form as well.
It should be noted however that most theories are naturally
axiomatised by conditions of the form $\varphi = 0$.

\begin{dfn}
  Let $T$ be a theory.
  A \emph{(pre-)model} of $T$ is an $\cL$-(pre-)structure $M$ in which
  $T$ is satisfied.
\end{dfn}
The notions of satisfaction and satisfiability of sets of conditions
from \fref{dfn:CondSat} apply in the special case of a theory (a
set of conditions without free variables).
In particular, a theory is satisfiable if and only if it has a model,
and by \fref{prp:cmplt}, this is the same as having a
pre-model.

A theory is \emph{complete} if it is satisfiable and maximal as such
(i.e., if it is a complete $0$-type), or at least if its set of
logical consequences is.
The complete theories are precisely those obtained as theories of
structures:
\begin{align*}
  \Th(M)
  & = \{\varphi = 0\colon \varphi \text{ an $\cL$-sentence and } \varphi^M = 0\} \\
  & \equiv \{\varphi = \varphi^M\colon \varphi \text{ an $\cL$-sentence}\}.
\end{align*}
(In the second line we interpret $\varphi = r$ as an abbreviation for a set of
conditions as described earlier.)

\subsection{Some examples of theories}
Using the metric, any equational theory (in the ordinary sense) can be
expressed as a theory, just replacing $x = y$ with $d(x,y) = 0$.
\begin{exm}
  Consider probability algebras (i.e., measure algebras, as discussed
  for example in \cite{Fremlin:MeasureTheoryVol3}, with total measure $1$).
  The language is $\cL = \{0,1,{}^c,\land,\lor,\mu\}$, with all continuity
  moduli being the identity.
  The theory of probability algebras, denoted $PrA$, consists of the
  following axioms:
  \begin{align*}
    & \langle\textit{equational axioms of Boolean algebras}\rangle\\
    & \mu(1) = 1\\
    & \mu(0) = 0\\
    & \forall xy\, \big( \mu(x)+\mu(y)=\mu(x\lor y)+\mu(x\land y) \big)\\
    & \forall xy\, \big( d(x,y)=\mu((x\land y^c)\lor(y\land x^c)) \big)
  \end{align*}
  The last two axioms are to be understood in the sense of
  \fref{ntn:UnivCond}.
  Thus $\forall xy\, \mu(x)+\mu(y)=\mu(x\lor y)+\mu(x\land y)$ should be understood as
  $\sup_{xy}\, |\half[\mu(x)+\mu(y)]-\half[\mu(x\lor y)+\mu(x\land y)]| = 0$, etc.
  In the last expression, division by two is necessary to keep the
  range in $[0,1]$.
  As we get used to this we will tend to omit it and simply write
  $\sup_{xy}\, |\mu(x)+\mu(y)-\mu(x\lor y)-\mu(x\land y)| = 0$.

  Note that we cannot express $\mu(x)=0\to x=0$, but we do not have to
  either: if $M$ is a model, $a \in M$, and $\mu^M(a) = 0$, then the axioms
  imply that $d^M(a,0^M) = 0$, whereby $a = 0^M$.

  The model companion of $PrA$ is $APA$, the theory of atomless
  probability algebras,
  which contains in addition the following sentence:
  \begin{align*}
    & \forall x\exists y\, \left( \mu(y\land x)=\half[\mu(x)] \right),
  \end{align*}
  Following \fref{rmk:AEntn} we can express this by:
  \begin{align*}
    & \sup_x \inf_y \left| \mu(y\land x)-\half[\mu(x)] \right| = 0
  \end{align*}
  (We leave it to the reader to verify that this sentential condition
  does indeed if and only if the probability algebra is atomless.)
\end{exm}

\begin{exm}[Convex spaces]
  Let us now consider a signature $\cL_{cvx}$ consisting of binary
  function symbols $c_\lambda(x,y)$ for all dyadic numbers $\lambda \in [0,1]$,
  which are all, say, $1$-Lipschitz in both arguments.
  Let $T_{cvx}$ consist of:
  \begin{align*}
    & (\forall xyz)\, d(z,c_\lambda(x,y)) \leq \lambda d(z,x) + (1-\lambda)d(z,y) \\
    & (\forall xyz)\, c_{\lambda_0+\lambda_1}(c_{\frac{\lambda_0}{\lambda_0+\lambda_1}}(x,y),z)
    = c_{\lambda_1+\lambda_2}(c_{\frac{\lambda_1}{\lambda_1+\lambda_2}}(y,z),x) && \lambda_0+\lambda_1+\lambda_2 = 1 \\
    & (\forall xyz)\, d(c_\lambda(x,z),c_\lambda(y,z)) = \lambda d(x,y).
  \end{align*}
  By \cite{mach:CnvxChr}, the models of $T_{cvx}$ are
  precisely the closed convex subspaces of Banach spaces of diameter
  $\leq 1$, equipped with the
  convex combination operations $c_\lambda(x,y) = \lambda x + (1-\lambda)y$.
  We may in fact restrict to a single function symbol $c_{1/2}(x,y)$,
  since every other convex combination operator with dyadic
  coefficients can be expressed using this single operator.
  Since our structures
  are by definition complete, dyadic convex combinations suffice.
\end{exm}

\begin{exm}
  \label{exm:BanachUnitBall}
  Let us continue with the previous example.
  We may slightly modify our logic allowing the distance symbol to
  have values in the compact interval $[0,2]$, so now models of
  $T_{cvx}$ are convex sets of diameter $\leq 2$.
  Let us add a constant symbol $0$, introduce $\|x\|$ as shorthand for
  $d(x,0)$ and $\lambda x$ as shorthand for $c_\lambda(x,0)$, and add the axioms:
  \begin{align*}
    & (\forall x)\, \|x\| \leq 1 \\
    & (\forall x\exists y) \, d(x,y/2)
    \land
    (1/2 \dotminus \|x\|) = 0
  \end{align*}
  The first axiom tells us that our model is a convex subset of the
  unit ball of the ambient Banach space.
  The second tells us that our model is precisely the unit ball:
  if $\|x\| \leq 1/2$ then $2x$ exists.
  One can add more structure on top of this, for example:
  \begin{enumerate}
  \item Multiplication by $i$, rendering the ambient space a complex
    Banach space.
  \item Function symbols $\lor$ and $\land$, rendering the ambient space a
    Banach lattice.
    In order to make sure we remain inside the unit ball, we actually
    need to add $(x,y) \mapsto (x\land y)/2$ and $(x,y) \mapsto (x\lor y)/2$ rather than
    $\lor$ and $\land$.
    In particular the model-theoretic study of independence in
    $L_p$ Banach lattices carried out in \cite{BenYaacov-Berenstein-Henson:LpBanachLattices} fits
    in this setting.
  \end{enumerate}
\end{exm}

Alternatively, in order to study a Banach space $E$ one could
introduce a multi-sorted structure where there is a sort $E_n$
for each closed ball around $0$ in $E$ of radius $n < \omega$.
On each sort, all the predicate symbols have values in a
compact interval, and operations such as $+$ go from
$E_n \times E_m$ to $E_{n+m}$.
However, since every sort $E_n$ is isomorphic to $E_1$ up to
rescaling, this boils down to the single-sorted approach described
above (and in particular, re-scaled addition is indeed the
convex combination operation).

It can be shown that either approach (single unit ball sort or
a sort for each radius) has the same power of expression as Henson's
logic, i.e., the translation from a Banach space structure \`a la
Henson to a unit ball structure (in an appropriate signature)
preserves such notions as elementary classes and extensions,
type-definable subsets of the unit ball, etc.
This should be intuitively clear from \fref{rmk:AEntn}.
Thus, continuous first order logic is indeed as good a setting for the
study of such properties as stability and independence in Banach space
structures as Henson's logic, but as we show later it is
much better adapted for such study.

The reader may find in \cite{BenYaacov:Perturbations} a treatment of
unbounded continuous signatures and structures,
an approach much closer
in spirit to Henson's treatment of Banach space structures.
It is proved there that approximate satisfaction of
positive bounded formulae (which makes sense in
any unbounded continuous signature) has the
same power of expression as satisfaction of conditions of
continuous first order logic.
The equivalence mentioned in the previous paragraph follows.
In addition, the single point compactification method defined there turns
such unbounded structures into bounded structures as studied here without
chopping them into pieces as above (and again, preserving such notions
as elementarity and definability).

\subsection{Type spaces of a theory}

If $T$ is a theory, we define its type spaces as in classical first
order logic:
\begin{align*}
  & \tS_n(T) = \{p \in \tS_n\colon T \subseteq p\} =
  \{\tp^M(\bar a)\colon M \models T \text{ and } \bar a \in M^n\}.
\end{align*}
This is a closed subspace of $\tS_n$, and therefore compact and
Hausdorff in the induced topology.
We define $[\varphi \leq r]^{\tS_n(T)} = [\varphi \leq r]^{\tS_n} \cap \tS_n(T)$, and
similarly for $\varphi = r$, etc.\
As before, we may omit $\tS_n(T)$ if the ambient type
space in question is clear.

If $M$ is a structure and $A \subseteq M$, we define $\cL(A)$ by adding
constant symbols for the elements of $A$ and identify $M$ with its
natural expansion to $\cL(A)$.
We define $T(A) = \Th_{\cL(A)}(M)$ and $\tS_n(A) = \tS_n(T(A))$, the
latter being the space of $n$-types over $A$.

This definition allows a convenient re-statement of the Tarski-Vaught
Test:
\begin{fct}[Topological Tarski-Vaught Test]
  \label{fct:TopTVTest}
  Let $M$ be a structure, $A \subseteq M$ a closed subset.
  Then the following are equivalent:
  \begin{enumerate}
  \item The set $A$ is (the domain of) an elementary substructure of
    $M$: $A \preceq M$.
  \item The set of realised types
    $\{\tp^M(a/A)\colon a \in A\}$ is dense in $\tS_1(A)$.
  \end{enumerate}
\end{fct}
\begin{proof}
  We use \fref{fct:TVTest}.
  Assume first that the set of realised types is dense.
  Let $\varphi(y,\bar a) \in \cL(A)$, $r = \inf_y \varphi(y,\bar a)^M \in [0,1]$.
  Then for every $\varepsilon > 0$ the set
  $[\varphi(y,\bar a) < r+\varepsilon] \subseteq \tS_1(A)$ is open and non-empty, so there is
  $b \in A$ such that $\varphi(b,\bar a)^M < r+\varepsilon$, whence condition (ii) of
  \fref{fct:TVTest}.
  Conversely, since the sets of the form $[\varphi(y,\bar a) < r]$ (with
  $\bar a \in A$)
  form a basis of open sets for $\tS_1(A)$, condition (ii) of
  \fref{fct:TVTest} implies the realised types are dense.
\end{proof}

By previous results, a uniform limit (or forced limit) of formulae
with parameters in $A$ is the same (as functions on $M$, or on any
elementary extension of $M$) as a continuous mapping
$\varphi\colon \tS_n(A) \to [0,1]$.
Such a definable predicate with parameters in $A$ is called a
definable predicate \emph{over $A$}, or an
\emph{$A$-definable} predicate.

We define $\kappa$-saturated and (strongly) $\kappa$-homogeneous structures
as usual, and show that every complete theory admits a monster model,
i.e., a $\kappa$-saturated and strongly $\kappa$-homogeneous model for some $\kappa$
which is far larger than the cardinality of any other set under
consideration.
It will be convenient to assume that there is always an ambient
monster model: every set of parameters we
consider is a subset of a monster model, and every model we consider
is an elementary substructure thereof.
(Even when considering an incomplete theory, each model of the theory
embeds in a monster model of its complete theory.)

If $\fM$ is a monster model and $A \subseteq \fM$ a (small) set, we
define
$$\Aut(\fM/A) = \{f \in \Aut(\fM)\colon f\rest_A = \id_A\}.$$
A definable predicate with parameters (in $\fM$) is
\emph{$A$-invariant} if it is fixed by all $f \in \Aut(\fM/A)$.

All type spaces we will consider in this paper are
quotient spaces of $\tS_n(\fM)$, where $\fM$ is a fixed monster
model of $T$, or of $\tS_n(T)$, which can be obtained using
the following general fact:
\begin{fct}
  \label{fct:FuncQuot}
  Let $X$ be a compact Hausdorff space and $\fA \subseteq C(X,[0,1])$ any
  sub-family of functions.
  Define an equivalence relation on $X$ by $x \sim y$ if $f(x) = f(y)$
  for all $f \in \fA$, and let $Y = X/{\sim}$.
  Then:
  \begin{enumerate}
  \item Every $f \in \fA$ factors uniquely through the quotient mapping
    $\pi \colon X \to Y$ as $f = f_Y \circ \pi$.
  \item The quotient topology on $Y$ is precisely the minimal topology
    under which every such $f_Y$ is continuous.
  \item This topology is compact and Hausdorff.
  \end{enumerate}

  Conversely, let $Y$ be a compact Hausdorff space and $\pi\colon X \to Y$ a
  continuous projection.
  Then $Y$ is a quotient space of $X$ and can be obtained as above
  using $\fA = \{f \circ \pi\colon f \in C(Y,[0,1])\}$.
\end{fct}
\begin{proof}
  The first item is by construction.
  Let $\sT_1$ be the quotient topology on $Y$ and $\sT_2$ the minimal
  topology in which every $f_Y$ is continuous.
  Then $\sT_1$ is compact as a quotient of a compact topology.
  If $y_1,y_2 \in Y$ are distinct then there is a function $f_Y$
  separating them, whereby $\sT_2$ is Hausdorff.
  Finally, let $V \subseteq [0,1]$ be open and $f \in \fA$,
  so $U = f_Y^{-1}(V) \subseteq Y$ is a pre-basic open set of $\sT_2$.
  Then $\pi^{-1}(U) = f^{-1}(V) \subseteq X$ is open, whereby $U \in \sT_1$.
  Thus $\sT_1$ refines $\sT_2$.
  Since $\sT_1$ is compact and $\sT_2$ is Hausdorff they must
  coincide.

  For the converse, the space $Y' = X/{\sim}$ constructed in this manner
  can be identified with $Y$.
  The original topology on $Y$ refines the quotient topology by the
  second item, and as above they must coincide.
\end{proof}

For example, let $\fM$ be a monster model and $A \subseteq \fM$ a set.
Then there is a natural projection $\pi\colon \tS_n(\fM) \to \tS_n(A)$
restricting from $\cL(\fM)$ to $\cL(A)$, and let:
\begin{align*}
  \fA & = \{\varphi \circ \pi\colon \varphi \in C(\tS_n(A),[0,1])\}, \\
  \fA' & = \{\varphi \in C(\tS_n(\fM),[0,1])\colon
  \text{ $\varphi$ is $A$-invariant}\}.
\end{align*}
Then clearly $\fA \subseteq \fA'$.
On the other hand, $\fA$ and $\fA'$
separate the same types, so by \fref{fct:FuncQuot} every $f \in
\fA'$ factors through $\tS_n(A)$ and $\fA' = \fA$.
In other words we've shown:
\begin{lem}
  \label{lem:InvarPred}
  Let $A$ be a set (in the monster model) and let $\varphi$ be an
  $A$-invariant definable predicate with parameters possibly outside
  $A$.
  Then $\varphi$ is (equivalent to) an $A$-definable predicate $A$.
\end{lem}

Let us adapt the notions of algebraicity and algebraic closure to
continuous logic:
\begin{lem}
  \label{lem:AlgTyp}
  Let $A$ be a set of parameters and $p(x) \in \tS_1(A)$.
  Then the following are equivalent:
  \begin{enumerate}
  \item For every $\varepsilon > 0$ there is a condition $(\varphi_\varepsilon(x) = 0) \in p$
    (with parameters in $A$) and $n_\varepsilon < \omega$ such that
    for every sequence $(a_i\colon i \leq n_\varepsilon)$, if
    $\varphi_\varepsilon(x_i) < 1/2$ for all $i \leq n_\varepsilon$ then $d(x_i,x_j) \leq \varepsilon$
    for all $i<j\leq n_\varepsilon$.
  \item Every model containing $A$ contains all realisations of $p$.
  \item Every indiscernible sequence in $p$ is constant.
  \item There does not exist an infinite sequence
    $(a_i\colon i < \omega)$ of realisations of $p$ such that
    $\inf \{d(a_i,a_j)\colon i < j < \omega\} > 0$.
  \item The set of realisations of $p$ is compact.
  \end{enumerate}
\end{lem}
\begin{proof}
  \begin{cycprf}
  \item[\impnext]
    We may assume that for the choice of $\varphi_\varepsilon$, the number $n_\varepsilon$ is
    minimal: we can therefore find in the universal domain
    elements $a_{<n_\varepsilon}$ such that $\varphi_\varepsilon(a_i) < 1/2$ for all $i < n_\varepsilon$
    and yet $d(a_i,a_j) > \varepsilon$ for all $i<j<n_\varepsilon$.
    Then we have (in the universal domain):
    $${\sup}_{x_{<n_\varepsilon}} \left(
      \bigwedge_{i<n_\varepsilon} \left( \half \dotminus \varphi_\varepsilon(x_i) \right) \land
      \bigwedge_{i<j<n_\varepsilon} \left( d(x_i,x_j) \dotminus \varepsilon \right) \right) > 0.$$
    Assume that $A \subseteq M$.
    Then the same holds in $M$, and we may therefore find $a_{<n_\varepsilon}$
    as above inside $M$.
    Assume also that $a \models p$.
    Then by assumption there is some $i < n_\varepsilon$ such that
    $d(a,a_i) \leq \varepsilon$.
    Since this holds for all $\varepsilon > 0$ and $M$ is complete we must have
    $a \in M$.
  \item[\impnext] By Downward L\"owenheim-Skolem.
  \item[\impnext] By (compactness and) Ramsey's theorem.
  \item[\impfirst]
    Since the set $\{d(x_i,x_j) \geq \varepsilon\colon i<j<\omega\} \cup \bigcup_{i<\omega} p(x_i)$ must be
    inconsistent.
  \item[(iv) $\Longleftrightarrow$ (v).] Condition (iv) is equivalent to saying that
    the set of realisations of $p$ is totally bounded.
    Since it is in addition automatically complete (a limit of
    realisations of $p$ is a realisation of $p$), this is the same as
    saying it is compact.
  \end{cycprf}
\end{proof}

\begin{dfn}
  \label{dfn:Alg}
  If $p \in \tS(A)$ satisfies any of the equivalent properties in
  \fref{lem:AlgTyp} then it is called \emph{algebraic}.
  We say that $a$ is \emph{algebraic} over $A$ if $\tp(a/A)$ is
  algebraic.
  We define the \emph{algebraic closure} of $A$, denoted $\acl(A)$, as
  the set of all algebraic elements over $A$.
  By \fref{lem:AlgTyp} if $A \subseteq M$ then $\acl(A) \subseteq M$ as well, so
  $\acl(A)$ is the same in every model containing $A$.
  If $A = \acl(A)$ then we say that $A$ is \emph{algebraically
    closed}.
\end{dfn}

\subsection{The metric on types of a complete theory}
Let $T$ be complete.
Since any two $n$-types are realised inside the monster model we can
define for every $p,q \in \tS_n(T)$:
\begin{align*}
  d(p,q) & = \inf \{d(\bar a,\bar b)\colon \bar a \models p \text{ and } \bar b
  \models q\}
\end{align*}
Here $d(a_{<n},b_{<n}) = \bigvee_{i<n} d(a_i,b_i)$.
It is trivial to verify this is a pseudo-metric.
By compactness the infimum is attained, so
it is in fact a metric: $d(p,q) = 0 \Longrightarrow p = q$.
Note also that we can
construct $\tS_n(T)$ as a set as $\fM^n/\Aut(\fM)$, where $\fM$ is the
monster model and we divide it by the action of its automorphism
group, in which case the distance between types in the one induced
from $\fM^n$.
(For this purpose, any $\omega$-saturated and strongly homogeneous
model of $T$ would serve just as well.)

The metric on $\tS_n(T)$ refines the logic topology.
Indeed, let $p \in \tS_n(T)$ and $U$ a neighbourhood of $p$.
Then there is a formula $\varphi$ such that
$p \in [\varphi = 0] \subseteq [\varphi < 1/2] \subseteq U$.
The uniform continuity of $\varphi$
implies the existence of $\delta > 0$ such that
$d(p,q) \leq \delta \Longrightarrow \varphi^q < 1/2$, so $U$ contains a metric neighbourhood of
$p$.
By a theorem of Henson (for Banach space structures in positive
bounded logic, but it boils down to the same thing), for a complete
countable theory $T$, the metric on $\tS_n(T)$ coincides with the
logic topology for all $n$ if and only if $T$
is separably categorical, i.e., if and only if it has a unique
separable model up to isomorphism.

Also, if $F \subseteq \tS_n(T)$ is closed and $\varepsilon > 0$, then so is the closed
$\varepsilon$-neighbourhood of $F$:
$$F^\varepsilon = \{p \in \tS_n(T)\colon d(p,F) \leq \varepsilon\}.$$
Indeed, since the set $F$ is closed in can be written as $[p(\bar x)]$
where $p$ is some partial type, and then
\begin{gather*}
  F^\varepsilon = [\exists \bar y\, \big( d(\bar x,\bar y) \leq \varepsilon \land p(\bar y) \big)].
\end{gather*}

This leads us to the following definition which will turn out to be
useful later on:
\begin{dfn}
  \label{dfn:TopoMetric}
  A compact \emph{topometric} space is a triplet $\langle X,\sT,d\rangle$,
  where $\sT$ is a
  compact Hausdorff topology and $d$ a metric on $X$, satisfying:
  \begin{enumerate}
  \item The metric refines the topology.
  \item For every closed $F \subseteq X$ and $\varepsilon > 0$, the closed
    $\varepsilon$-neighbourhood of $F$ is closed in $X$ as well.
  \end{enumerate}
\end{dfn}

When dealing with topometric spaces some care must be taken about
the language.
We will follow the convention that
whenever we use terms which come from the realm of general topology
(such as compactness, closed and open sets, etc.)\ we refer to the
topology.
When wish to refer to the metric, we will use terminology
that clearly comes from the realm of metric spaces.
When there may be ambiguity, we will say explicitly to which part
we are referring.

We may therefore sum up the previous observations as:
\begin{fct}
  The type space $\tS_n(T)$ is a compact topometric space.
\end{fct}

We will come back to topometric spaces later.
Let us now conclude with a small fact about them:
\begin{lem}
  Let $X$ be a compact topometric space (so by the terminological
  convention above, we mean to say that the topology is compact).
  Then it is complete (as a metric space).
\end{lem}
\begin{proof}
  Let $(x_n\colon n < \omega)$ be a Cauchy sequence in $X$.
  We may assume that
  $d(x_n,x_{n+1}) \leq 2^{-n-1}$ for all $n$.
  For each $n$ the set $\{x_n\}^{2^{-n}}$, (the closed
  $2^{-n}$-ball around $x_n$) is closed in the topology, and
  $(\{x_n\}^{2^{-n}}\colon n < \omega)$ is a decreasing sequence of non-empty
  closed sets.
  By compactness there is some $x$ in the intersection, and clearly
  $x_n \to x$ in the metric.
\end{proof}

\subsection{Quantifier elimination}

\begin{dfn}
  A \emph{quantifier-free definable predicate} is a definable
  predicate defined by a forced limit of quantifier-free formulae.

  A theory has \emph{quantifier elimination} if every formula can
  be uniformly approximated over all models of $T$ by quantifier-free
  formulae, i.e., if every formula is equal in models of $T$
  to a quantifier-free definable predicate.

  (In order to avoid pathologies when there are no constant symbols
  in $\cL$, we must allow that if $\varphi$ is a
  formula without free variables, the quantifier-free definable
  predicate may have a free variable.)
\end{dfn}

We introduce the following criterion for quantifier elimination,
analogous to the classical back-and-forth criterion:
\begin{dfn}
  \label{dfn:BnF}
  We say that a theory $T$ has the \emph{back-and-forth} property
  if for every two $\omega$-saturated models $M,N \models T$, non-empty tuples
  $\bar a \in M^n$ and $\bar b \in N^n$, and singleton $c \in M$,
  if $\bar a$ and $\bar b$ have the same
  quantifier-free type (i.e., $\varphi^M(\bar a) = \varphi^N(\bar b)$ for all
  quantifier-free $\varphi$) then there is $d \in N$ such that $\bar a,c$ and
  $\bar b,d$ have the same quantifier-free type.
\end{dfn}

\begin{thm}
  \label{thm:BnFQE}
  The following are equivalent for any continuous theory
  $T$ (not necessarily complete):
  \begin{enumerate}
  \item The theory $T$ admits quantifier elimination.
  \item The theory $T$ has the back-and-forth property.
  \end{enumerate}
\end{thm}
\begin{proof}
  Assume first that $T$ admits quantifier elimination.
  Then under the assumptions we have $\bar a \equiv \bar b$.
  Let $p(x,\bar y) = \tp(c,\bar a)$.
  Then $p(x,\bar b)$ is consistent and is realised by some $d \in N$ by
  $\omega$-saturation.

  For the converse we introduce an auxiliary definition:
  The $\inf_y$-type of a tuple $\bar a \in M$ is given by the function
  $\varphi(\bar x) \mapsto \varphi^M(\bar a)$, where $\varphi$ varies over all the formulae of
  the form $\inf_y \psi(y,\bar x)$, $\psi$ quantifier-free.
  We define $\tS_n^{\inf_y}(T)$ as the set of all $\inf_y$-types
  of $n$-tuples in models of $T$.
  This is a quotient of $\tS_n(T)$, and we equip it with the quotient
  topology, which is clearly compact and Hausdorff.

  We claim first that if $M,N \models T$ and $\bar a \in M^n$, $\bar b \in N^n$
  have the same quantifier-free type then they have the same
  $\inf_y$-type.
  Since we may embed $M$ and $N$ elementarily in more saturated
  models, we may assume both are $\omega$-saturated.
  Assume $\inf_y \psi(y,\bar a) = r$.
  Then there are $c_m \in M$ such that $\psi(c_m,\bar a) \leq r+2^{-m}$, and
  by $\omega$-saturation there is $c \in M$ such that
  $\psi(c,\bar a) = r$.
  Therefore there is $d \in N$ such that $\psi(d,\bar b) = r$, whereby
  $\inf_y \psi^M(y,\bar a) \geq \inf_y \psi^N(y,\bar b)$.
  By a symmetric argument we have equality.

  We conclude that the quantifier-free formulae separate points in
  $\tS_n^{\inf_y}(T)$.
  Since quantifier-free formulae form a family of continuous functions on
  $\tS^{\inf_y}_n(T)$ which is closed under continuous connectives,
  the quantifier-free formulae are dense in
  $C(\tS_n^{\inf_y}(T),[0,1])$ by \fref{cor:SWCnct}.
  In particular, every $\inf_y$-formula can be uniformly approximated
  by quantifier-free formulae, and by induction on the structure of
  the formula, every formula can be
  thus approximated (on models of $T$, or equivalently on
  $\tS_n(T)$).
\end{proof}

\begin{cor}
  The theory of atomless probability algebras (described above)
  is complete and has quantifier elimination.
\end{cor}
\begin{proof}
  The back-and-forth property between complete atomless
  probability algebras is immediate, and does not require
  $\omega$-saturation (or rather, as
  it turns out, all complete atomless probability
  algebras are $\omega$-saturated).
  Indeed, two $n$-tuples $a_{<n} \in M$ and $b_{<n} \in N$
  have the same quantifier-free type ($\bar a \equiv^{qf} \bar b$)
  if and only if they generate isomorphic algebras.
  Letting $a^0 = a$, $a^1 = a^c$ and
  $\bar a^{\bar t} = \bigwedge_{i<n} a_i^{t_i}$ for $\bar t \in \{0,1\}^n$ we have
  $\bar a \equiv^{qf} \bar b$ if and only if
  $\mu(\bar a^{\bar t}) = \mu(\bar b^{\bar t})$ for all
  $\bar t \in \{0,1\}^n$.
  If $c \in M$ is any other singleton use atomlessness to find for
  each $\bar t \in \{0,1\}^n$ an event
  $d_{\bar t} \leq \bar b^{\bar t}$ such that
  $\mu(d_{\bar t}) = \mu(c \land \bar a^{\bar t})$.
  Then $d = \bigvee d_{\bar t}$ will do.

  Also, every two probability algebras can be embedded in a third one,
  which can be further embedded in a complete atomless one,
  whence the completeness.
\end{proof}

\begin{dfn}
  A theory $T$ is \emph{model complete} if for every $M,N \models T$,
  $M \subseteq N \Longrightarrow M \prec N$.
\end{dfn}

We leave the following as an exercise to the reader (see
\cite{BenYaacov:NakanoSpaces} for a complete proof):
\begin{prp}
  A theory $T$ is model complete if and only if every formula can be
  uniformly approximated on $\tS_n(T)$ by formulae of the form
  $\inf_{\bar y} \varphi(\bar x,\bar y)$, where $\varphi$ is quantifier-free.
\end{prp}

\subsection{Continuous first order logic and open Hausdorff cats}

We now show the equivalence between the framework of
continuous first order logic and that of (metric) open Hausdorff cats.
For this we assume familiarity with the latter framework, as exposed
in \cite{BenYaacov:Morley}.
The reader who is not familiar with open Hausdorff cats may safely
skip this part.

To every theory $T$ we associate its type-space functor $\tS(T)$ in
the usual manner.
For every $n < \omega$ we defined $\tS_n(T)$ above.
If $m,n < \omega$ and $f\colon n \to m$ is any mapping, we define
$f^*\colon \tS_m(T) \to \tS_n(T)$ by
$f^*(p(x_{<m})) =
\{\varphi(x_{<n}) = \varphi(x_{f(0)},\ldots,x_{f(n-1)})^p\colon \varphi(x_{<n}) \in \cL\}$, i.e.,
$f^*\colon \tp(a_0,\ldots,a_{m-1}) \mapsto \tp(a_{f(0)},\ldots,a_{f(n-1)})$.

\begin{fct}
  \label{fct:CFOIsOpenHausCat}
  Let $T$ be a continuous first order theory.
  Then its type-space functor $\tS(T)$ is an open, compact and
  Hausdorff type-space functor in the sense of \cite{BenYaacov:PositiveModelTheoryAndCats}.

  Since a type-space functor is one way to present a cat, this can be
  restated as:
  every continuous first order theory is an open Hausdorff cat.
\end{fct}
\begin{proof}
  Clearly $\tS(T)$ is a Hausdorff compact type-space functor.
  To see it is open, let $\pi_n\colon \tS_{n+1}(T) \to \tS_n(T)$ consist of
  restriction to the $n$
  first variables (so $\pi_n = (n \hookrightarrow n+1)^*$).
  Let $U \subseteq \tS_{n+1}(T)$ be
  a basic open set, i.e., of the form $[\varphi(\bar x,y) < r]$.
  Then $\pi_n(U) = [\inf_y \varphi < r]$ is open as well, so $\pi_n$ is an open
  mapping.
\end{proof}

Recall that a \emph{definable $n$-ary function} from a cat $T$ to a
Hausdorff space $X$ is a continuous mapping $f\colon \tS_n(T) \to X$.
Equivalently, this is a mapping from the
models of $T$ to $X$ such
that for every closed subset $F \subseteq X$, the property
$f(\bar x) \in F$ is type-definable without parameters (whence
\emph{definable} function).
A \emph{definable metric} is a definable binary function which defines
a metric on the models.

Note that $d$ is indeed a definable metric, so $T$ is a metric cat, and
the models of $T$ (in the sense of continuous first-order logic) are
precisely its complete models as a metric Hausdorff cat, as
defined in \cite{BenYaacov:Morley}.

For the converse, we will use the following property of definable
functions in open cats:
\begin{lem}
  \label{lem:OpenHausQuant}
  Let $T$ be an open cat.
  Let $f(\bar x,y)$ be an definable $n+1$-ary function from $T$ to
  $[0,1]$ (or the reals, for that matter), and let
  $g(\bar x) = \sup_y f(\bar x,y)$.
  Then $g$ is also a definable function.
\end{lem}
\begin{proof}
  For every real number, we can define $g(\bar x) \geq r$ by the
  partial type $\bigwedge_{s<r} \exists y\, f(\bar x,y) \geq s$.
  We can also define $g(\bar x) \leq r$ by $\forall y\,f(\bar x,y) \leq r$,
  and this is expressible by a partial type since $T$ is assumed to
  be open.
\end{proof}

\begin{rmk}
  This can also be stated in purely topological terms:
  Let $X$ and $Y$ be compact Hausdorff spaces, and $f\colon X \to Y$ an open
  continuous surjective mapping.
  Let $\varphi\colon X \to [0,1]$ be continuous, and let
  $\psi\colon Y \to [0,1]$ be defined by $\psi(y) = \sup \{\varphi(x)\colon f(x) = y\}$.
  Then $\psi$ is continuous.
\end{rmk}

Observe that if  $T$ is a metric open Hausdorff cat and $d$ a definable
metric on some sort, then by compactness the metric is bounded.
Thus, up to rescaling we may always assume its range is contained in
$[0,1]$.

\begin{thm}
  \label{thm:OpenHausCats}
  Let $T$ be a metric open Hausdorff cat, and let $d$ be a definable
  metric on the home sort with range in $[0,1]$.

  Then there exists a metric signature $\hat \cL$ whose distinguished
  metric symbol is $\hat d$, and an $\hat \cL$-theory
  $\hat T$, such that $\tS(\hat T) \simeq \tS(T)$, and such that the
  metric $\hat d$ coincide with $d$.

  Moreover, if $\kappa$ is such that $\tS_n(T)$ has a basis of cardinality
  $\leq \kappa$ for all $n < \omega$, then we can arrange that $|\cL| \leq \kappa$.
\end{thm}
\begin{proof}
  For each $n$, we choose a family $F_n \subseteq C(\tS_n(T),[0,1])$ which
  separates points, are closed under $\{\lnot,\dotminus,\frac{x}{2}\}$, and
  such that $d \in F_2$.
  By \fref{lem:OpenHausQuant} we may further assume that
  $\sup_y P(\bar x,y) \in F_n$ for each $P(\bar x,y) \in F_{n+1}$.
  We can always choose the $F_n$ such that $|F_n| \leq \kappa$ for all $n$,
  where $\kappa$ is as in the moreover part.
  By \fref{cor:SWCnct}, $F_n$ is dense in $C(\tS_n(T),[0,1])$.

  We observe that every $P \in F_n$ is uniformly continuous with respect
  to $d$.
  Indeed, $|P(\bar x) - P(\bar y)|$ is a continuous function from
  $\tS_{2n}(T)$, and for every $\varepsilon > 0$, the following partial type is
  necessarily inconsistent:
  $$\{|P(\bar x) - P(\bar y)| \geq \varepsilon\}
  \cup \{ d(x_i,y_i) \leq 2^{-m}\colon m<\omega,i<n\}.$$
  Therefore there is $m < \omega$ such that
  $d(\bar x,\bar y) \leq 2^{-m} \Longrightarrow |P(\bar x) - P(\bar y)| \leq \varepsilon$.

  Let $\hat \cL_n = \{\hat P\colon n < \omega, P \in F_n\}$, where we associate to
  each $n$-ary predicate symbol $\hat P$ the uniform continuity moduli
  obtained in the previous paragraph.
  Every universal domain of $T$, or closed subset thereof, is
  naturally a $\hat \cL$-structure, by interpreting each $\hat P$ as
  $P$.
  In particular, all the predicates satisfy the appropriate
  continuity moduli.

  Clearly, the family of $n$-ary definable functions is closed under
  continuous connectives.
  Also, if $\varphi(x_{<n})\colon \tS_n(T) \to [0,1]$ is a definable $n$-ary
  function and $f\colon n\to m$ is any mapping, then
  $\varphi(x_{i_0},\ldots,x_{i_{n-1}}) = \varphi \circ f^*$ is a definable $m$-ary
  function: in other words, the definable functions are closed under
  changes of variables.
  Finally, by \fref{lem:OpenHausQuant}, the definable functions
  are closed under quantification.
  Put together, every $\hat \cL$-formula $\varphi(x_{<n})$ induces a
  definable function $\varphi \in C(\tS_n(T),[0,1])$.
  Since $F_n$ is dense there, for every $m$ we can find
  $P_{\varphi,m} \in F_n$ such that $|P_{\varphi,m} - \varphi| \leq 2^{-m}$.
  Let $\hat T$ consist of:
  \begin{align*}
    & \sup_{\bar x} |\hat P_{\varphi,m}(\bar x) - \varphi(\bar x)| \leq 2^{-m}
    && \varphi, P_{\varphi,m} \text{ as above} \\
    & \sup_{\bar x} \hat P \leq 2^{-m}
    && P \in F_n \text{ and } P(\tS_n(T)) \subseteq[0,2^{-m}]
  \end{align*}
  Clearly, every model of $T$, viewed as an $\hat \cL$-structure, is a
  model of $\hat T$.

  We claim that $\tS(\hat T) \approx \tS(T)$.
  Indeed, for every $p \in \tS(T)$, define $\theta_n(p)$ by
  $\varphi^{\theta_n(p)} = \varphi(p)$ for every $\hat \cL$-formula $\varphi(x_{<n})$.
  Then $\theta_n(p)$ is just the $\hat \cL$-type of any realisation of $p$
  (again, viewing models of $T$ as $\hat \cL$-structure), so in
  particular is it indeed in $\tS_n(\hat T)$.
  Since the predicate symbols of $\hat \cL$ separate $T$-types,
  $\theta_n$ is injective.
  It is also clearly continuous, and therefore (as its domain is
  compact and its target Hausdorff) is a homeomorphism of $\tS_n(T)$
  with a closed subspace of $\tS_n(\hat T)$.
  If $\theta_n$ is not surjective, there is $q \in \tS_n(\hat T)$, and a
  neighbourhood $q \in U$ which is disjoint of the image of $\theta_n$.
  As usual there is a formula $\varphi$ such that
  $q \in [\varphi = 0] \subseteq [\varphi < 1/2] \subseteq U$, and letting $\psi = \lnot(\varphi \dotplus \varphi)$:
  $q \in [\psi = 1] \subseteq [0 < \psi] \subseteq U$, so $\psi(\tS_n(T)) = \{0\}$.
  Then $\hat T$ says that
  $\sup_{\bar x} |\psi(\bar x) - \hat P_{\psi,2}(\bar x)| \leq 1/4$
  and $\sup_{\bar x} \hat P_{\psi,2}(\bar x) \leq 1/4$, which in turn
  imply that $\psi \leq 1/2$ and therefore $\varphi \geq 1/4$, so $q$ is inconsistent
  with $\hat T$.

  Thus $\theta_n\colon \tS_n(T) \to \tS_n(\hat T)$ is a homeomorphism for every
  $n$, and by construction it is compatible with the functor structure
  so $\theta\colon \tS(T) \to \tS(\hat T)$ is the required homeomorphism of
  type-space functors.

  It is clear from the way we axiomatised it that $\hat T$ has
  quantifier elimination.
\end{proof}

\fref{thm:OpenHausCats}, combined \fref{fct:CFOIsOpenHausCat},
says that framework of continuous first order theories coincides with
that of metric open Hausdorff cats.
In fact, we know from \cite{BenYaacov:Morley} that if $T$ is a non-metric
Hausdorff cat then its home sort can be ``split'' into (uncountably
many) hyperimaginary metric sorts, so in a sense every Hausdorff cat
is metric.
Thus, with some care, this observation can be generalised to all open
Hausdorff cats.

Compact type-space functors are a structure-free and
language-free way of
presenting cats.
By \cite{BenYaacov:PositiveModelTheoryAndCats}, one can associate to each compact type-space
functor a positive Robinson theory $T'$ in some language $\cL'$ and
talk about universal domains for that cat as $\cL'$-structures.
Then ``complete models'' mean the same thing in both settings:

\begin{fct}
  Let $T$ be a complete continuous theory and $T'$ the corresponding
  cat (i.e., positive Robinson theory in a language $\cL'$).

  Then there is a one-to-one correspondence between
  \hbox{($\kappa$-)monster} models of
  $T$ and \hbox{($\kappa$-)universal} domains of $T'$ over the same set of
  elements (for a fixed big cardinal $\kappa$), in such a way that the type of a
  tuple in one is the same as its type in the other.

  Moreover, a closed subset of such a model is an elementary submodel in the
  sense of $T$ if and only if it is a complete submodel
  in the sense of $T'$, as defined in \cite{BenYaacov:Morley}.
\end{fct}
\begin{proof}
  For the first condition, the identification of the
  type-space functors of $T$ and of $T'$ imposes for each monster
  model (or universal domain) of one an interpretation of the
  other language, and it is straightforward (though tedious) to
  verify then that the new structure is indeed a universal domain (or
  a monster model) for the other theory.

  The moreover part is a special case of \fref{fct:TopTVTest}.
\end{proof}

Since $T$ is a metric cat, its monster model is a monster metric
space (momspace) as defined in \cite{Shelah-Usvyatsov:837},
and its models (in
the sense of continuous logic) are precisely the class $K_1^c$
studied there. So results proved for momspaces apply in our context.

\section{Imaginaries}
\label{sec:Img}

In classical first order model theory there are two common ways to
view (and define) imaginaries: as canonical parameters for formulae,
or, which is more common, as classes modulo definable equivalence
relations.
Of course, any canonical parameters for a formula can be viewed as an
equivalence class, and an equivalence class is a canonical parameter
for the formula defining it, so both approaches are quite equivalent
in the discrete setting.

We have already observed that in the passage from discrete to
continuous logic equivalence relations are replaced with
pseudo-metrics.
On the other hand, the notion of a canonical parameter remains
essentially the same: the canonical parameter for $\varphi(\bar x,\bar a)$
is something (a tuple, an imaginary\ldots) $c$ which an automorphism fixes if
and only if it does not alter the meaning of the formula
(i.e., $c = f(c) \Longleftrightarrow \varphi(\bar x,\bar a) \equiv \varphi(x,f(\bar a))$ for every
$f \in \Aut(\fM)$).

As in the classical setting, both approaches are essentially
equivalent, but in practise the canonical parameter
approach has considerable advantages.
In particular, when doing stability, we would need to consider
canonical parameters for definable predicates $\psi(\bar x,A)$, which only
has finitely many free variables but may have infinitely many
parameters.
Canonical parameters for such definable predicates are dealt with as
with canonical parameters for formulae, and the existence of
infinitely many parameters introduces very few additional;
complications.
On the other hand, if we wished to define the canonical parameter as
an equivalence class
modulo a pseudo-metric we would be forced to consider pseudo-metrics on
infinite tuples, the logic for whose equivalence classes could become
messy.

Other minor advantages include the fact that we need not ask
ourselves whether a particular formula defines a pseudo-metric on every
structure or only on models of a given theory, and finally the
conceptually convenient fact
that unlike equivalence relations which need to be replaced with
pseudo-metrics, canonical parameters are a familiar notion which we
leave unchanged.

Let $\cL$ be a continuous signature.
For convenience, assume that $\cL$ has a single sort $S$.
If we wanted to work with a many-sorted language we would have to keep
track on which variables (in the original language) belong to which
sort, but other than that the treatment is identical.

Let us start with the case of a formula $\varphi(x_{<n},y_{<m})$, where
$x_{<n}$ are the free variables, and $y_{<m}$ are
the parameter variables (and to simplify notation we will consider
that this partition of the variables is part of $\varphi$).
We define a new signature $\cL_\varphi$, consisting of $\cL$ along with a
new sort $S_\varphi$ for the canonical parameters for instances
$\varphi(\bar x,\bar a)$ of $\varphi$, and the following new symbols:
\begin{enumerate}
\item A predicate symbol $d_\varphi(z,z')$ on $S_\varphi^2$ which will be
  the distance symbol for $S_\varphi$.
\item A predicate symbol $P_\varphi(x_{<n},z)$ on $S^n\times S_\varphi$.
  Its uniform continuity moduli with respect to the first $n$
  arguments are the same as $\varphi$'s with respect to $x_{<n}$, and with
  respect to the last argument it is the identity.
  (This symbol is not strictly necessary, but will be convenient.)
\end{enumerate}

We will expand every $\cL$-structure $M$ to an $\cL_\varphi$-structure,
interpreting $S_\varphi$ as the family the canonical parameters of all
instances $\varphi(\bar x,\bar a)$ of $\varphi(\bar x,\bar y)$ in $M$.
We first expand $M$ into an $\cL_\varphi$-pre-structure $M_{\varphi,0}$ by defining:
\begin{align*}
  & S_\varphi^{M_{\varphi,0}} = M^m \\
  & P_\varphi^{M_{\varphi,0}}(\bar a,(\bar b)) = \varphi^M(\bar a,\bar b) \\
  & d_\varphi^{M_{\varphi,0}}((\bar b),(\bar b')) =
  {\sup}_{\bar x}\,|\varphi(\bar x,\bar b)-\varphi(\bar x,\bar b')|^M
\end{align*}
(Here $\bar b$ is an $m$-tuple in the home sort of $M$, $(\bar b)$ is
the corresponding element of $S_\varphi^{M_{\varphi,0}}$.)
We leave the reader the verification that $d_\varphi^{M_{\varphi,0}}$ is a
pseudo-metric and the uniform continuity
moduli fixed above are indeed respected, so $M_{\varphi,0}$ is an
$\cL_\varphi$-pre-structure.
We then define $M_\varphi = \widehat{M_{\varphi,0}}$, the structure associated to
the pre-structure $M_{\varphi,0}$.

For $\bar b \in M^m$, let $[\bar b]_\varphi$ denote the
image of $(\bar b)$ in $S_\varphi^{M_\varphi}$.
Clearly, every automorphism of $M$ extends uniquely to an automorphism
of $M_\varphi$, and it fixes $\varphi(\bar x,\bar b)$ if and only if it fixes
$[\bar b]_\varphi$, which is therefore a canonical parameter for
$\varphi(\bar x,\bar b)$.
If $c = [\bar b]_\varphi$ then $P_\varphi(\bar x,c)$ coincides with $\varphi(\bar x,\bar
b)$.
We therefore allow ourselves to abuse notation and denote either one
simply as $\varphi(\bar x,c)$.

The properties of the new sort are described intuitively by the
following axioms:
\begin{gather*}
  \forall zz'\,
  \big( d_\varphi(z,z') = {\sup}_{\bar x}|P_\varphi(\bar x,z)-P_\varphi(\bar x,z')| \big), \\
  \forall z\,\exists\bar y\,\forall\bar x\,
  \big( \varphi(\bar x,\bar y) = P_\varphi(\bar x,z) \big), \\
  \forall\bar y\,\exists z\,\forall\bar x\,
  \big( \varphi(\bar x,\bar y) = P_\varphi(\bar x,z) \big).
\end{gather*}
To get the precise axioms we apply \fref{rmk:AEntn},
so``$\forall\exists\forall (\ldots=\ldots)$''  should be read as
``$\sup \inf \sup |\ldots-\ldots| = 0$'', etc.
We therefore define $T_\varphi$ to be the following $\cL_\varphi$-theory:
\begin{gather*}
  \sup_{zz'}\, \big|
  d_\varphi(z,z') - {\sup}_{\bar x}|P_\varphi(\bar x,z)-P_\varphi(\bar x,z')|
  \big| = 0, \\
  \sup_z \inf_{\bar y} \sup_{\bar x}
  \big| \varphi(\bar x,\bar y) - P_\varphi(\bar x,z) \big| = 0, \\
  \sup_{\bar y} \inf_z \sup_{\bar x}
  \big| \varphi(\bar x,\bar y) - P_\varphi(\bar x,z) \big| = 0.
\end{gather*}

One easily verifies that:
\begin{prp}
  An $\cL_\varphi$-structure is a model of $T_\varphi$ if and only if it is of the
  form $M_\varphi$ for some $\cL$-structure $M$.
\end{prp}
Therefore, if $T$ is a complete $\cL$-theory then $T\cup T_\varphi$ is a
complete $\cL_\varphi$-theory.
We discussed the case of a single formula $\varphi$, but we can do the same
with several (all) formulae simultaneously.

\begin{rmk}
  \label{rmk:PsMetImg}
  As we said earlier, the continuous analogue of an equivalence
  relation is a pseudo-metric.
  We can recover classes modulo pseudo-metrics from canonical
  parameters in very straightforward manner:
  \begin{enumerate}
  \item Assume that $\varphi(\bar x,\bar y)$ defines a pseudo-metric on
    $M^n$.
    Then the pseudo-metric $d_\varphi^{M_{\varphi,0}}$
    (defined on $S_\varphi^{M_{\varphi,0}} = M^n$) coincides with $\varphi^M$,
    and $(S_\varphi,d_\varphi)^{M_\varphi}$ is the completion of the set of equivalence
    classes of $n$-tuples modulo the relation $\varphi(\bar a,\bar b)$,
    equipped with the induced metric.
  \item In particular, let $\xi_n(x_{<n},y_{<n}) = \bigvee_{i<n} d(x_i,y_i)$.
    Then the sort $S_{\xi_n}$ is the sort of $n$-tuples with the
    standard metric.
  \end{enumerate}
\end{rmk}

We now wish to define canonical parameters to definable predicates of
the form $\psi(\bar x,B)$, i.e. to instances of $\psi(\bar x,Y)$, which may
unavoidably have infinitely many parameters (or parameter variables).
We can write $\psi(\bar x,Y)$ as a uniform limit of a sequence of
formulae $(\varphi_n(\bar x,\bar y_n)\colon n < \omega)$, where $(\bar y_n)$ is an
increasing sequence of tuples and $Y = \bigcup_n \bar y_n$.
We may further assume that the rate of convergence is such that
$|\varphi_n - \psi| \leq 2^{-n}$.
Since each $\bar y_n$ is finite, we may assume that $|Y| = \omega$.

We define $\cL_\psi$ by adding to $\cL$ a new sort $S_\psi$
and predicate symbols $d_\psi(z,z')$ and $P_\psi(\bar x,z)$ as
before.

Given a structure $M$ we construct $M_{\psi,0}$ much the same as
before:
\begin{align*}
  & S_\psi^{M_{\psi,0}} = M^\omega \\
  & P_\psi^{M_{\psi,0}}(\bar a,(B)) = \psi^M(\bar a,B)
  \, (= \lim \varphi_n^M(\bar a,\bar b_n)) \\
  & d_\psi^{M_{\psi,0}}((B),(B')) =
  {\sup}_{\bar x}\,|\psi(\bar x,B)-\psi(\bar x,B')|^M
\end{align*}
Again $M_{\psi,0}$ is a pre-structure, and we define
$M_\psi = \widehat{M_{\psi,0}}$.
Letting $[B]_\psi$ be the image of $(B)$ in $M_\psi$, we see again that
$[B]_\psi$ is a canonical parameter for $\psi(\bar x,B)$.
If $c = [B]_\psi$ then again we allow ourselves the abuse of notation
which consists of denoting either of $\psi(\bar x,B)$ and $P_\psi(\bar x,c)$
(which are equivalent) by $\varphi(\bar x,c)$.

Notice that $\cL_\psi$ is a finitary language in which the canonical
parameters are singletons, and the fact that
$Y$ is an infinite tuple is indeed hidden.

The theory $T_\psi$, the analogue of $T_\varphi$ above,
consists of
infinitely many axioms, according to the following schemes, which we
interpret according to \fref{rmk:AEntn} as above:
\begin{gather*}
  \forall zz'\,\big( d_\psi(z,z') =
  {\sup}_{\bar x}|P_\psi(\bar x,z)-P_\psi(\bar x,z')| \big), \\
  \forall z\,\exists\bar y_n\,\forall\bar x\,
  |\varphi_n(\bar x,\bar y_n) - P_\psi(\bar x,z)| \leq 2^{-n}, \\
  \forall y_n\,\exists\bar z\,\forall\bar x\,
  |\varphi_n(\bar x,\bar y_n) - P_\psi(\bar x,z)| \leq 2^{-n}
\end{gather*}
Again, the models of $T_\psi$ are precisely the $\cL_\psi$-structures of the
form $M_\psi$.

\begin{rmk}
  It is not true that if $T$ is model complete
  then so is $T \cup T_\varphi$.

  In the case of canonical parameters for a single formula $\varphi$
  we can remedy this deficiency as in discrete logic by
  naming the mapping $\bar b \mapsto [\bar b]_\varphi$ with a new
  function symbol $\pi_\varphi\colon S^m \to S_\varphi$.
  We leave it to the reader to verify that,
  if $T$ is a model complete $\cL$-theory then $T \cup T'_\varphi$ is a
  model complete $\cL_\varphi'$-theory, where
  \begin{gather*}
    \cL_\varphi' = \cL_\varphi \cup \{\pi_\varphi\}, \qquad
    T'_\varphi = T_\varphi \cup \{\forall\bar x\bar y\, \big(
    \varphi(\bar x,\bar y) = P_\varphi(\bar x,\pi_\varphi(\bar y))
    \big)\}.
  \end{gather*}

  The \emph{graph} of $\pi_\varphi$ (or of any function of continuous
  structures) is defined here to be the predicate
  $\gamma_\varphi(\bar y,z) = d_\varphi(\pi_\varphi(\bar y),z)$.
  This predicate is definable in $\cL_\varphi$:
  $$\gamma_\varphi(\bar y,z)
  = {\sup}_{\bar x} |\varphi(\bar x,\bar y) - P_\varphi(\bar x,z)|.$$
  It follows that the addition of
  $\pi_\varphi$ to the language does not add any structure
  (see \cite[Section~1]{BenYaacov:DefinabilityOfGroups}
  for a more detailed discussion
  of definable functions in continuous logic).

  In the case of canonical parameters of
  $\psi(\bar x,Y) = \lim \varphi_n(\bar x,\bar y_n)$, which are quotients of
  infinite tuples, we cannot add a function
  symbol $\pi_\varphi$.
  Instead we observe that in the case of a single formula it would
  have sufficed to name $\gamma_\varphi$ by a predicate (rather than naming
  $\pi_\varphi$).
  While $\gamma_\psi$ would depend infinitely many variables and thus still
  impossible to add to the language, we may add finite approximations.
  We add predicate symbols $\gamma_{\varphi_n,\psi}(\bar y_n,z)$
  and add to $T_\psi$ the axioms:
  \begin{align*}
    \forall\bar y_nz\, \big( \gamma_{\varphi_n,\psi}(\bar y_n,z) =
    {\sup}_{\bar x} |\varphi_n(\bar x,\bar y_n) - P_\psi(\bar x,z)| \big).
  \end{align*}
  Call the expanded language $\cL_\psi'$ and the expanded theory $T'_\psi$.
  Again, we leave it to the reader to verify that if $T$ is a
  model complete $\cL$-theory, then so is $T \cup T'_\psi$ as an
  $\cL_\psi'$-theory.

  We leave the details to the interested reader.
\end{rmk}

\section{Local types and $\varphi$-predicates}
\label{sec:LocTyp}

In this section and later
we will consider formulae whose free variables are split in two groups
$\varphi(\bar x,\bar y)$.
Following \fref{rmk:PsMetImg}, we may replace finite tuples of
variables with single ones, and therefore allow ourselves to restrict
our attention to formulae of the form $\varphi(x,y)$, where $x$ and $y$ may
belong to distinct sorts.

We will associate variable letters with sorts: $x$, $x_i$,
etc., belong to one sort, $y$, $y_j$, etc., to another, and so forth.
Accordingly, we will denote the difference spaces of types in the
variable $x$ by $\tS_x(T)$, $\tS_x(A)$, etc.

We now fix a formula $\varphi(x,y)$.
We will define spaces of $\varphi$-types as quotients
of spaces we have already constructed using \fref{fct:FuncQuot}.
The notion of a $\varphi$-type over a model is fairly straightforward.
A few more steps will be required in order to obtain the correct
notion of a $\varphi$-type over an arbitrary set.

\begin{dfn}
  Let $M$ be a structure.
  We define $\tS_\varphi(M)$ as the quotient of $\tS_x(M)$ given by the
  family of functions
  $\sA_{M,\varphi} = \{\varphi(x,b)\colon b \in M\}$.
  An element of $\tS_\varphi(M)$ is called a
  \emph{(complete) $\varphi$-type} over $M$.

  Accordingly, the \emph{$\varphi$-type of $a$ over $M$}, denoted
  $\tp_\varphi(a/M)$, is given by the mappings $\varphi(x,b) \mapsto \varphi(a,b)$ where $b$
  varies over all elements of the appropriate sort of $M$.
\end{dfn}

We equip $\tS_\varphi(M)$ with a metric structure:
For $p,q \in \tS_\varphi(M)$ we define
$$d(p,q) = \sup_{b \in M} |\varphi(x,b)^p - \varphi(x,b)^q|.$$
\begin{fct}
  Equipped with this metric and with its natural topology (inherited
  as a quotient space from $\tS_x(M)$),
  $\tS_\varphi(M)$ is a compact topometric space as in
  \fref{dfn:TopoMetric}.
\end{fct}

\begin{dfn}
  Let $M$ be a structure.
  A \emph{$\varphi$-predicate} over $M$, or an \emph{$M$-definable
    $\varphi$-predicate}, is a continuous mapping
  $\psi\colon \tS_\varphi(M) \to [0,1]$.
\end{dfn}

\begin{fct}
  Let $\psi\colon \tS_x(M) \to [0,1]$ be an $M$-definable predicate.
  Then the following are equivalent:
  \begin{enumerate}
  \item $\psi$ is a $\varphi$-predicate (i.e., factors through the projection
    $\tS_x(M) \to \tS_\varphi(M)$).
  \item There are formulae $\psi_n(x,\bar b_n)$, each obtained using
    connectives from several instances $\varphi(x,b_{n,j})$, where each
    $b_{n,j} \in M$, and in $M$ we have $\psi(x) = \flim \psi_n(x)$.
  \item $\psi$ can be written as $f \circ (\varphi(x,b_i)^M\colon i < \omega)$ where
    $f\colon [0,1]^\omega \to [0,1]$ is continuous and $b_i \in M$ for all $i < \omega$.
  \end{enumerate}
\end{fct}
\begin{proof}
  \begin{cycprf}
  \item[\impnext] Standard application of
    \fref{cor:SWCnct}.
  \item[\impnfirst] Immediate.
  \end{cycprf}
\end{proof}

\begin{lem}
  \label{lem:ModelPhiType}
  Let $\fM$ be a monster model, and $M \preceq \fM$ a model.
  Let $\psi(x)$ be an $M$-invariant $\varphi$-predicate over $\fM$.
  Then $\psi$ is (equal to) a $\varphi$-predicate over $M$.
\end{lem}
\begin{proof}
  We know that $\psi(x)$ is equal to a definable predicate over $M$, so
  $\psi(x) = \lim \psi_n(x,c_n)$ where each $\psi_n(x,z_n)$ is a formula and
  $c_n \in M$.
  We also know that $\psi(x) = \lim \chi_n(x,d_n)$, where each
  $\chi_n(x,\bar d_n)$ is a combination of instances $\varphi(x,d_{n,j})$ with
  parameters $d_{n,j} \in \fM$.
  For all $\varepsilon > 0$ there exists $n < \omega$ such that:
  \begin{align*}
    \fM & \models {\sup}_x |\psi_n(x,c_n)-\chi_n(x,\bar d_n)| < \varepsilon
    \intertext{In particular:}
    \fM & \models {\inf}_{\bar y_n} {\sup}_x
    |\psi_n(x,c_n)-\chi_n(x,\bar y_n)| < \varepsilon
    \intertext{Since $M \preceq \fM$:}
    M & \models {\inf}_{\bar y_n} {\sup}_x
    |\psi_n(x,c_n)-\chi_n(x,\bar y_n)| < \varepsilon
    \intertext{So there is $\bar d_n' \in M$ such that:}
    M & \models  {\sup}_x |\psi_n(x,c_n)-\chi_n(x,\bar d_n')| < \varepsilon.
  \end{align*}
  We can therefore express $\psi$ as $\lim \chi_n(x,\bar d_n')$, which is a
  $\varphi$-predicate over $M$.
\end{proof}

This leads to the following:
\begin{dfn}
  Let $A$ be a (small) set in a monster model $\fM$.
  \begin{enumerate}
  \item A \emph{$\varphi$-predicate over $A$}, or an \emph{$A$-definable
      $\varphi$-predicate} is a $\varphi$-predicate over
    $\fM$ which is $A$-invariant.
  \item We define $\tS_\varphi(A)$ as the quotient of $\tS_\varphi(\fM)$
    determined by the $A$-definable $\varphi$-predicates.
    The points of this space are called \emph{(complete) $\varphi$-types
      over $A$}.
  \item Accordingly, the \emph{$\varphi$-type} $\tp_\varphi(a/A)$ is given by
    the mappings $\psi(x) \mapsto \psi(a)$ where $\psi$ varies over all $A$-definable
    $\varphi$-predicates.
  \end{enumerate}
\end{dfn}

By \fref{lem:ModelPhiType}, these definitions coincide with
previous ones in case that $A = M \preceq \fM$.

We conclude with a result about compatibility of two kinds of
extensions of local types: to the algebraic closure of the
set of parameters, and to more (all) formulae.
\begin{lem}
  \label{lem:LocAclGlobCompat}
  Let $p \in \tS_\varphi(A)$, and let $q \in \tS_\varphi(\acl(A))$ and $r \in \tS_x(A)$
  extend $p$.
  Then $q\cup r$ is consistent.
\end{lem}
\begin{proof}
  Let $R_1 \subseteq \tS_x(\acl(A))$ be the pullback of $r$ (i.e., the set of
  all its extensions to a complete type over $\acl(A)$), and
  $R_2 \subseteq \tS_\varphi(\acl(A))$ the image of $R_1$ under the restriction
  projection  $\tS_x(\acl(A)) \to \tS_\varphi(\acl(A))$.
  Then $R_2$ is the set of all extensions of $p$ to $\acl(A)$
  compatible with $r$, and we need to show that $q \in R_2$ (i.e., that
  $R_2$ is the set of all the extensions of $p$).

  Indeed, assume not.
  The sets $R_1$ and therefore $R_2$ are closed.
  Therefore we can separate $R_2$ from $q$ by a $\varphi$-predicate
  $\psi(x,a)$, with parameter $a \in \acl(A)$, such that $\psi(x,a)^q = 0$ and
  $R_2 \subseteq [\psi(x,a) = 1]$.
  Since $a \in \acl(A)$, by \fref{lem:AlgTyp}
  there is a sequence $(a_i\colon i<\omega)$ such that:
  \begin{enumerate}
  \item Every $a_i$ is an $A$-conjugate of $a$.
  \item For every $\varepsilon > 0$ there is $n$ such that every
    $A$-conjugate of $a$ is in the $\varepsilon$-neighbourhood of some $a_i$
    for $i < n$.
  \end{enumerate}
  Define $\psi_n(x,a_{<n}) = \bigwedge_{i<n} \psi(x,a_i)$.
  Then, by uniform continuity of $\psi(x,y)$ with respect to $y$, the
  sequence $(\psi_n(x,a_{<n})\colon n < \omega)$ converges uniformly
  to the predicate
  $\chi(x) = \inf \{\psi(x,a')\colon a' \equiv_A a\}$.
  Thus $\chi(x)$ is a definable $\varphi$-predicate (as a limit of such) and
  $A$-invariant, so it is an $A$-definable $\varphi$-predicate.
  On the one hand we clearly have
  $\chi(x)^p = \chi(x)^q = 0$.
  On the other, as $R_2$ is $A$-invariant as well,
  we have $R_2 \subseteq [\psi(x,a_i) = 1]$ for all $i < \omega$, so
  $\chi(x)^p = \chi(x)^r = 1$.
  This contradiction concludes the proof.
\end{proof}

\begin{lem}
  \label{lem:LocAclTrans}
  Let $A \subseteq M$ where $M$ is strongly $(|A|+\omega)^+$-homogeneous, and
  $p \in \tS_\varphi(A)$.
  Then $\Aut(M/A)$ acts transitively on the extensions of $p$ to
  $\tS_\varphi(\acl(A))$.
\end{lem}
\begin{proof}
  Follows from (and is in fact equivalent to)
  \fref{lem:LocAclGlobCompat}.
\end{proof}

\section{Local stability}
\label{sec:LocStab}

Here we answer C.\ Ward Henson's question mentioned in the
introduction.
Throughout this section $T$ is a fixed continuous theory
(not necessarily complete) in a signature $\cL$.

\begin{dfn}
  \begin{enumerate}
  \item We say that a formula $\varphi(x,y)$ is \emph{$\varepsilon$-stable} for a real
    number $\varepsilon > 0$ if in models of $T$ there is no
    infinite sequence $(a_ib_i\colon i < \omega)$ satisfying for all
    $i < j$: $|\varphi(a_i,b_j) - \varphi(a_j,b_i)| \geq \varepsilon$.
  \item We say that $\varphi(x,y)$ is \emph{stable} if it is $\varepsilon$-stable for
    all $\varepsilon > 0$.
  \end{enumerate}
\end{dfn}

\begin{lem}
  \label{lem:EStableEquiv}
  Let $\varphi(x,y)$ be a formula, $\varepsilon > 0$.
  Then the following are equivalent:
  \begin{enumerate}
  \item The formula $\varphi$ is $\varepsilon$-stable.
  \item It is impossible to find
    $0 \leq r < s \leq 1$ and an infinite sequence $(a_ib_i\colon i < \omega)$ such
    that $r \leq s-\varepsilon$ and for all $i<j$:
    $\varphi(a_i,b_j) \leq r$, $\varphi(a_j,b_i)\geq s$.
  \item There exists a natural number
    $N$ such that in model of $T$ there is no
    finite sequence $(a_ib_i\colon i < N)$ satisfying:
    \begin{gather*}
      \label{eq:EStableEquivInc}
      \tag{$*$}
      \text{for all } i < j < k: \qquad
      |\varphi(a_j,b_i) - \varphi(a_j,b_k)| \geq \varepsilon.
    \end{gather*}
  \end{enumerate}
\end{lem}
\begin{proof}
  \begin{cycprf}
  \item[\eqnext] Left to right is immediate.
    For the converse assume $\varphi$ is not $\varepsilon$-stable, and let
    the sequence $(a_ib_i\colon i < \omega)$ witness this.
    For every $\delta > 0$ we can find using
    Ramsey's Theorem arbitrarily long sub-sequences
    $(a_i'b_i'\colon i < N)$ such that in addition:
    \begin{gather*}
      \text{If } i < j \text{ and } i' < j'
      \text{ then: }
      |\varphi(a_i,b_j) - \varphi(a_{i'},b_{j'})|,|\varphi(a_j,b_i) - \varphi(a_{j'},b_{i'})| \leq \delta.
    \end{gather*}
    (For this we use the classical finite Ramsey's Theorem.
    We could also use the infinite version to obtain a single infinite
    sequence with the same properties.)

    By compactness we can find an infinite sequence
    $(c_id_i\colon i <\omega)$ witnessing $\varepsilon$-instability such that
    in addition, for $i < j$,
    $\varphi(c_i,d_j) = r$ and $\varphi(c_j,d_i) = s$ do not depend on $i,j$.
    Thus $|r-s| \geq \varepsilon$.
    If $r < s$ we are done.
    If $r > s$ we can reverse the ordering on all the finite
    subsequences obtained above, thus exchanging $r$ and $s$, and
    conclude in the same manner.
  \item[\eqnext]
    Now right to left is immediate.
    For left to write, we argue as above, using Ramsey's Theorem and
    compactness, that there exists an infinite sequence
    $(a_ib_i\colon i < \omega)$ such that for $i < j < k$ we have
    $|\varphi(a_j,b_i) - \varphi(a_j,b_k)| \geq \varepsilon$ and
    $\varphi(a_i,b_j) = r$ and $\varphi(a_j,b_i) = s$ do not depend on $i,j$.
    Then again $|r-s| \geq \varepsilon$ and we conclude as above.
  \end{cycprf}
\end{proof}

It follows that stability is a symmetric property: define
$\tilde \varphi(y,x) \eqdef \varphi(x,y)$; then $\varphi$ is \hbox{($\varepsilon$-)stable} if and
only if $\tilde \varphi$ is.

\begin{ntn}
  \label{ntn:NPhiEpsilon}
  If $\varphi$ is $\varepsilon$-stable we define $N(\varphi,\varepsilon)$ to be the minimal $N$ such
  that no sequence $(a_ib_i\colon i < N + 1)$ exists satisfying
  \fref{eq:EStableEquivInc}.
\end{ntn}

Let us define the \emph{median value} connective
$\med_n\colon [0,1]^{2n-1} \to [0,1]$:
$$\med_n(t_{<2n-1}) = \bigwedge_{w \in [2n-1]^n} \bigvee_{i \in w} t_i
= \bigvee_{w \in [2n-1]^n} \bigwedge_{i \in w} t_i.$$
If $\varphi(x,y)$ is $\varepsilon$-stable define:
\begin{align*}
  d^\varepsilon\varphi(y,x_{<2N(\varphi,\varepsilon)-1})
  & = \med_{N(\varphi,\varepsilon)}\big( \varphi(x_i,y) \colon i <2N(\varphi,\varepsilon)-1 \big).
\end{align*}

\begin{lem}
  \label{lem:AprxPhiDef}
  Let $M$ be a model and $p \in \tS_\varphi(M)$.
  Then there exist $c^\varepsilon_{<2N(\varphi,\varepsilon)-1} \in M$ such that, for
  every $b \in M$:
  \begin{gather*}
    \left|\varphi(x,b)^p - d^\varepsilon\varphi(b,c^\varepsilon_{<2N(\varphi,\varepsilon)-1})\right| \leq \varepsilon.
  \end{gather*}
\end{lem}
\begin{proof}
  We argue as in the proof of \cite[Lemma~2.2]{Pillay:GeometricStability}.
  Choose a realisation $c \models p$ in the monster model: $c \in \fM \succeq M$.
  We construct by induction on $n$, tuples $c_n \in M$ in the sort of
  $x$, an increasing sequence of sets $K(n) \subseteq \bP(\omega)$, and tuples
  $a_w \in M$ in the sort of $y$ for each $w \in K(n)$, as follows.

  At the $n$th step we assume we have already chosen $c_{<n}$.
  We define:
  \begin{gather*}
    K(n) = \left\{ w \subseteq n\colon \exists a \in M \text{ such that }
      |\varphi(c,a) - \varphi(c_i,a)| > \varepsilon \text{ for all } i \in w \right\}.
  \end{gather*}
  For every $w \in K(n)$ such that $a_w$ has not
  yet been chosen, choose $a_w \in M$ witnessing that
  $w \in K(n)$.
  Note that if $w \subseteq m < n$ and $a_w$ witnesses that
  $w \in K(m)$ then it also witnesses that
  $w \in K(n)$, so there is no problem keeping previously
  made choices.
  We now have:
  \begin{align*}
    & \forall w \in K(n), i \in w: && |\varphi(c_i,a_w) - \varphi(c,a_w)| > \varepsilon,
  \end{align*}
  whereby:
  \begin{gather*}
    {\sup}_x  \left(
      \mathop{\bigwedge_{w \in K(n)}}_{i \in w} |\varphi(c_i,a_w) - \varphi(x,a_w)|
    \right) > \varepsilon.
  \end{gather*}
  This holds in $\fM$; but since all of the parameters of the form
  $a_w,b_w,c_i$ are in $M$, the last inequality actually
  holds in $M$.
  Therefore there exists $c_n \in M$ such that:
  \begin{align*}
    & \forall w \in K(n), i \in w: && |\varphi(c_i,a_w) - \varphi(c_n,a_w)| > \varepsilon.
  \end{align*}
  This concludes the $n$th step of the construction.

  Note that if $w \in K(n)$ and $m < n$ then
  $w\cap m \in K(n)$ as well.

  \begin{clm}
    For all $n$ and $w\in K(n)$: $|w| < N(\varphi,\varepsilon)$.
  \end{clm}
  \begin{clmprf}
    If not there is $w = \{m_0 < \ldots < m_{N-1}\} \in K(n)$ where
    $N \geq N(\varphi,\varepsilon)$.
    Define $m_N = n$ (so $m_{N-1} < m_N$), and for $j < N$, let
    $w_j = \{m_i\colon i<j\}$.
    Then for all $i<j<k\leq N$ we have $m_i \in w_j \in K(m_k)$, whereby:
    \begin{gather*}
      \left| \varphi(c_{m_i},a_{w_j})-\varphi(c_{m_k},a_{w_j}) \right| > \varepsilon.
    \end{gather*}
    Thus the sequence $(c_{m_i},a_{w_i}\colon i< N+1)$ contradicts the choice
    of $N(\varphi,\varepsilon)$.
  \end{clmprf}

  It follows that for every $w \in [2N(\varphi,\varepsilon)-1]^{N(\varphi,\varepsilon)}$ and $a \in M$:
  \begin{gather*}
    \bigwedge_{i\in w} \varphi(c_i,a) -\varepsilon \leq \varphi(c,a) \leq \bigvee_{i\in w} \varphi(c_i,a)+\varepsilon
  \end{gather*}
  Whereby $|\varphi(c,a) - d^\varepsilon\varphi(a,c_{<2N(\varphi,\varepsilon)-1})| \leq \varepsilon$, as required.
\end{proof}

\begin{dfn}
  Let $p(x) \in \tS_\varphi(M)$.
  A \emph{definition} for $p$ is an $M$-definable predicate
  $\psi(y)$ satisfying $\varphi(x,b)^p = \psi^M(b)$ for all $b \in M$.
  If such a definable predicate exists then it is unique
  (any two such definable predicates coincide on $M$, and therefore on
  every elementary extension of $M$), and is
  denoted $d_p\varphi(y)$.
\end{dfn}

Assume now that $\varphi(x,y)$ is stable, and let:
\begin{gather*}
  X = (x^n_i\colon n < \omega, i <2N(\varphi,2^{-n})-1), \\
  d\varphi(y,X) = \flim_n  d^{2^{-n}} \varphi(y,x^n_{<2N(\varphi,2^{-n})-1}).
\end{gather*}

\begin{prp}
  \label{prp:PhiDef}
  Let $M$ be a model, and $p \in \tS_\varphi(M)$.
  Then there are parameters $C \subseteq M$ such that
  $d\varphi(y,C) = d_p\varphi(y)$ (so in particular, a definition $d_p\varphi$ exists).
  Moreover, $d_p\varphi$ is an $M$-definable $\tilde \varphi$-predicate.
\end{prp}
\begin{proof}
  For each $n < \omega$ choose
  $c^n_{<2N(\varphi,2^{-n})-1}$ as in \fref{lem:AprxPhiDef},
  and let $C = (c^n_i\colon n < \omega,i <2N(\varphi,2^{-n})-1)$.

  Let $\xi\colon M \to [0,1]$ be defined as $b \mapsto \varphi(x,b)^p$.
  Then $|d^{2^{-n}}\varphi(y,c_{n,<2N(\varphi,2^{-n})-1})^M - \xi| \leq 2^{-n}$,
  whereby:
  \begin{gather*}
    \xi = \flim_n d^{2^{-n}}\varphi(y,c^n_{<2N(\varphi,2^{-n})-1})^M = d\varphi(y,C)^M.
  \end{gather*}
  This precisely means that $d\varphi(y,C) = d_p\varphi$.

  That $d\varphi(x,C)$ is a $\tilde \varphi$-predicate follows from its
  construction.
\end{proof}

From this point onwards we assume that $\cL$ has a sort for the
canonical parameters of instances of $d\varphi(y,X)$ for every stable
formula $\varphi(x,y) \in \cL$.
If not, we add these sorts as in \fref{sec:Img}.
It should be pointed out that if $M$ is an $\cL$-structure and
$\|M\| \geq |\cL|$, the addition of the new sorts does not change $\|M\|$:
this can be seen directly from the construction, or using the Downward
L\"owenheim-Skolem Theorem (\fref{fct:DLS}) and the fact that
we do not change $|\cL|$.

For every stable formula and type $p \in \tS_\varphi(M)$ we define $\Cb_\varphi(p)$
as the canonical
parameter of $d_p\varphi(y)$.
With the convention above we have $\Cb_\varphi(p) \in M$.
Notice that if $p,q \in \tS_\varphi(M)$, $c = \Cb_\varphi(p)$ and $c' = \Cb_\varphi(q)$,
then $d(c,c')$ (in the sense of the sort of canonical parameters for
$d\varphi$) is equal to $d(p,q)$ in $\tS_\varphi(M)$.

As with structures, we will measure the size of a type space
$\tS_\varphi(M)$ by its metric density character $\|\tS_\varphi(M)\|$.

\begin{prp}
  \label{prp:LocStabCnt}
  The following are equivalent for a formula $\varphi(x,y)$:
  \begin{enumerate}
  \item $\varphi$ is stable.
  \item For every $M \models T$, every $p \in \tS_\varphi(M)$ is definable.
  \item For every $M \models T$, $\|\tS_\varphi(M)\| \leq \|M\|$.
  \item There exists $\lambda \geq |T|$ such that whenever $M \models T$ and
    $\|M\| \leq \lambda$ then $\|\tS_\varphi(M)\| \leq \lambda$ as well.
  \end{enumerate}
\end{prp}
\begin{proof}
  \begin{cycprf}
  \item[\impnext] By \fref{prp:PhiDef}.
  \item[\impnext] Let $D \subseteq M_{d\varphi}$ be the family of canonical parameters of
    instances of $d\varphi(y,X)$ which actually arise as definitions of
    $\varphi$-types over $M$.
    Then $\|D\| \leq \|M\|$, and $D$ is isometric to $\tS_\varphi(M)$.
  \item[\impnext] Immediate.
  \item[\impfirst]
    Let $\lambda \geq |T|$ be any cardinal and assume $\varphi$ is unstable.
    It is a classical fact that there exists a linear order $(I,<)$
    of cardinality $\lambda$
    admitting $>\lambda$ initial segments:
    for example, let $\mu$ be the least cardinal
    such that $2^\mu > \lambda$ and let $I = \{0,1\}^{<\mu}$ equipped with the
    lexicographic ordering.

    Assuming $\varphi$ is unstable then we can find
    (using \fref{lem:EStableEquiv} and compactness)
    $0 \leq r < s \leq 1$ and  a sequence $(a_ib_i\colon i \in I)$ such that
    $i < j$ imply $\varphi(a_i,b_j) \leq r$ and $\varphi(a_j,b_i) \geq s$.
    By the Downward L\"owenheim-Skolem Theorem there exists a model
    $\{b_i\colon i \in I\} \subseteq M \preceq \fM$ such that $\|M\| \leq \lambda$.
    On the other hand, by compactness,
    for every initial segment $C \subseteq I$ there exists
    $a_C$ such that
    $\varphi(a_C,b_i) \geq s$ if $i \in C$
    and
    $\varphi(a_C,b_i) \leq r$ if $i \notin C$.
    Let $p_C = \tp_\varphi(a_C/M)$.

    If $C,C'$ are two distinct initial segments of $I$ then
    $d(p_C,p_{C'}) \geq s-r$, showing that
    $\|\tS_\varphi(M)\| > \lambda$ and concluding the proof.
  \end{cycprf}
\end{proof}

\begin{dfn}
  Let $(X,d)$ be a metric space.
  The \emph{diameter} of a subset $C \subseteq X$ is defined as
  \begin{gather*}
    \diam(C) = \sup\{d(x,y)\colon x,y \in C\}
  \end{gather*}
  We say that a subset $C \subseteq X$ is \emph{$\varepsilon$-finite} if it can be
  written as $C = \bigcup_{i<k} C_i$, where $\diam(C_i) \leq \varepsilon$ for all
  $i < k$.
  In this case, its \emph{$\varepsilon$-degree}, denoted $\deg_\varepsilon(C)$, is the
  minimal such $k$.
\end{dfn}

Note that if $C$ and $C'$ are $\varepsilon$-finite then so is $C\cup C'$ and
$\deg_\varepsilon(C\cup C') \leq \deg_\varepsilon(C) + \deg_\varepsilon(C')$.

\begin{dfn}
  Let $X$ be a compact topometric space.

  For a fixed $\varepsilon > 0$, we define a decreasing sequence of closed
  subsets $X_{\varepsilon,\alpha}$ by induction:
  \begin{align*}
    X_{\varepsilon,0} & = X\\
    X_{\varepsilon,\alpha} & = \bigcap_{\beta<\alpha} X_{\varepsilon,\beta} \qquad\qquad \text{($\alpha$ a limit ordinal)}\\
    X_{\varepsilon,\alpha+1} & = \bigcap \{F \subseteq X_{\varepsilon,\alpha}\colon \text{$F$ is closed and }
    \diam(X_{\varepsilon,\alpha}\setminus F) \leq \varepsilon\}\\
    X_{\varepsilon,\infty} & = \bigcap_\alpha X_{\varepsilon,\alpha}
  \end{align*}
  Finally, for any non-empty subset $C \subseteq X$ we define its
  $\varepsilon$-Cantor-Bendixson rank in $X$ as:
  \begin{gather*}
    \CB_{X,\varepsilon}(C) = \sup \{\alpha\colon C\cap X_{\varepsilon,\alpha} \neq \emptyset\} \in Ord \cup \{\infty\}
  \end{gather*}
  If $\CB_{X,\varepsilon}(C) < \infty$ we also define $\CBm_{X,\varepsilon}(C) = C \cap
  X_{\varepsilon,\CB_{X,\varepsilon}(C)}$, i.e., the set of points of maximal rank.
\end{dfn}

It is worthwhile to point out that
either $X_{\varepsilon,\alpha} \neq \emptyset$ for every $\alpha$ (and eventually stabilises
to $X_{\varepsilon,\infty}$) or there is a maximal $\alpha$ such that
$X_{\varepsilon,\alpha} \neq \emptyset$.
The same holds for the sequence $\{C\cap X_{\varepsilon,\alpha}\colon \alpha\in Ord\}$ if $C \subseteq X$ is
closed.

Assume that $C \subseteq X$ is closed and $\alpha = \CB_{X,\varepsilon}(C) < \infty$.
Then by the previous paragraph
$C$ contains points of maximal rank, i.e., $\CBm_{X,\varepsilon}(C) \neq \emptyset$.
Moreover, $\CBm_{X,\varepsilon}(C)$ is compact and admits in $X_{\varepsilon,\alpha}$ an open
covering by sets of diameter $\leq\varepsilon$.
By compactness, it can be covered by finitely many such, and is
therefore $\varepsilon$-finite.
This need not necessarily hold in case $C$ is not closed.

We will use this definition for $X = \tS_\varphi(M)$, where $M$ is at least
$\omega$-saturated.
In this case we may write $\CB_{\varphi,M,\varepsilon}$ instead of
$\CB_{\tS_\varphi(M),\varepsilon}$, etc.

\begin{rmk}
  In the definition of the $\varepsilon$-Cantor-Bendixson rank we defined
  $X_{\varepsilon,\alpha+1}$ by removing from $X_{\varepsilon,\alpha}$ all its ``small open subsets'',
  i.e., its open subsets of diameter $\leq\varepsilon$.
  There exist other possible definitions for the $\varepsilon$-Cantor-Bendixson
  derivative, using different notions of smallness.
  Such notions are studied in detail in
  \cite{BenYaacov:TopometricSpacesAndPerturbations}
  where
  it is shown that in the end they all boil down to the same thing.
\end{rmk}

\begin{prp}
  $\varphi$ is stable if and only if for one (any) $\omega$-saturated model
  $M \models T$: $\CB_{\varphi,M,\varepsilon}(\tS_\varphi(M)) < \infty$ for all $\varepsilon$.
\end{prp}
\begin{proof}
  If not, let $Y = \{p \in \tS_\varphi(M)\colon \CB_{\varphi,M,\varepsilon} = \infty\}$.
  Then $Y$ is compact, and if $U \subseteq Y$ is relatively open and
  non-empty then $\diam(U) > \varepsilon$.
  We can therefore find non-empty open sets $U_0,U_1$ such that $\bar
  U_0,\bar U_1 \subseteq U$ and $d(U_0,U_1) > \varepsilon$.
  Proceed by induction.
  This would contradict stability of $\varphi$ in a countable fragment of
  the theory.

  The converse is not really important, and is pretty standard.
\end{proof}

From now on we assume that $\varphi$ is stable.

\begin{dfn}
  Let $M \models T$ and $A \subseteq M$, and assume that
  $M$ is $(|A|+\omega)^+$-saturated and strongly homogeneous.
  A subset $F \subseteq \tS_\varphi(M)$ is \emph{$A$-good} if it is:
  \begin{enumerate}
  \item Metrically compact.
  \item Invariant under automorphisms of $M$ fixing $A$.
  \end{enumerate}
\end{dfn}

Recall the notions of algebraicity and algebraic closure from
\fref{dfn:Alg}.
\begin{lem}
  \label{lem:GoodDef}
  Assume that $F \subseteq \tS_\varphi(M)$ is $A$-good.
  Then every $p \in F$ is definable over $\acl(A)$.
\end{lem}
\begin{proof}
  We know that $p$ is definable, so let $d\varphi(y,C)$ be its
  definition, and $c = \Cb_\varphi(p)$ the canonical parameter of the
  definition.
  We may write $d\varphi(x,C)$ as $d_p\varphi(y,c)$.

  Assume that $c \notin \acl(A)$.
  Then there exists
  an infinite sequence $(c_i\colon i < \omega)$ in
  $\tp(c/A)$ such that $d(c_i,c_j) \geq \varepsilon > 0$ for all $i < j$, and we
  can realise this sequence in $M$ by the saturation assumption.
  By the homogeneity assumption, each $d_\varphi p(y,c_i)$ defines a type
  $p_i$ which is an $A$-conjugate of $p$.
  Therefore $p_i \in F$ for all $i < \omega$, and $d(p_i,p_j) \geq \varepsilon$
  for all $i < j < \omega$, contradicting metric compactness.
\end{proof}

\begin{lem}
  \label{lem:goodexist}
  Assume that $A \subseteq M \models T$, $M$ is $(|A|+\omega)^+$-saturated and strongly
  homogeneous, and $F \subseteq \tS_\varphi(M)$ is closed, non-empty, and
  invariant under $\Aut(M/A)$.
  Then $F$ contains an $A$-good subset.
\end{lem}
\begin{proof}
  Define by induction on $n$: $F_0 = F$, and
  $F_{n+1} = \CBm_{\varphi,M,2^{-n}}(F_n)$.
  Then $(F_n\colon n < \omega)$ is a decreasing sequence of non-empty closed
  subsets of $\tS_\varphi(M)$, and so $F_\omega = \bigcap_n F_n \neq \emptyset$.
  The limit set $F_\omega$ is $\varepsilon$-finite
  for every $\varepsilon > 0$, i.e., it is totally bounded.
  Since the metric refines the topology, $F_\omega$ is also metrically
  closed in $\tS_\varphi(M)$ and thus complete.
  We see that $f_\omega$ is a totally bounded complete metric space and
  therefore metrically compact.
  Also, each of the $F_n$ is invariant under $\Aut(M/A)$, and so is
  $F_\omega$.
\end{proof}

\begin{prp}
  \label{prp:Ext}
  Let $A \subseteq M \models T$, and let $p \in \tS_\varphi(A)$.
  Then there exists $q \in \tS_\varphi(M)$ extending $p$ which is definable
  over $\acl(A)$.
\end{prp}
\begin{proof}
  We may replace $M$ with a larger model, so we might as well assume
  that $M$ is $(|A|+\omega)^+$-saturated and strongly homogeneous.
  Let $P = \{q \in \tS_\varphi(M)\colon p \subseteq q\}$.
  By \fref{lem:goodexist}, there is an $A$-good subset $Q \subseteq P$,
  which is non-empty by definition.
  By \fref{lem:GoodDef}, any $q \in Q$ is an $\acl(A)$-definable
  extension of $p$.
\end{proof}

\begin{prp}
  \label{prp:LocSym}
  Let $M \models T$, $p(x) \in \tS_\varphi(M)$ and $q(y) \in \tS_{\tilde \varphi}(M)$.
  Let $d_p\varphi(y)$ and $d_q{\tilde \varphi}(x)$ be their respective
  definitions, and
  recall that these are a $\tilde \varphi$- and a $\varphi$-predicate,
  respectively.
  Then $d_p\varphi(y)^q = d_q{\tilde \varphi}(x)^p$.
\end{prp}
\begin{proof}
  Let $M_0 = M$.
  Given $M_n \succeq M$, obtain $p_n \in \tS_\varphi(M_n)$ and
  $q_n \in \tS_{\tilde \varphi}(M_n)$ by applying the definition of $p$ and
  $q$, respectively, to $M_n$ (these are indeed complete satisfiable
  $\varphi$- and $\tilde \varphi$-types).
  Realise them by $a_n$ and $b_n$, respectively, in some extension
  $M_{n+1} \succeq M_n$.
  Repeat this for all $n < \omega$.

  We now have for all $i < j$: $\varphi(a_j,b_i) = d_p\varphi(y)^q$ and
  $\varphi(a_i,b_j) = d_q{\tilde \varphi}(x)^p$, and if these differ we get a
  contradiction to the stability of $\varphi$.
\end{proof}

\begin{prp}
  \label{prp:LocStat}
  Assume that $A \subseteq M$ is algebraically closed, $p,p' \in \tS_\varphi(M)$ are both
  definable over $A$, and $p\rest_A = p'\rest_A$.
  Then $p = p'$.
\end{prp}
\begin{proof}
  Let $b \in M$, $q = \tp_{\tilde \varphi}(b/A)$.
  By \fref{prp:Ext} there is $\hat q \in \tS_{\tilde \varphi}(M)$
  extending $q$ which is definable over $\acl(A) = A$, and let
  $d_{\hat q}{\tilde \varphi}(x)$ be this definition.
  Recalling that $d_p\varphi$ and $d_{p'}\varphi$ are $\tilde \varphi$-predicates,
  $d_{\hat q}{\tilde \varphi}$ is a $\varphi$-predicate, and all of them are over
  $A$, we have:
  \begin{align*}
    \varphi(x,b)^p
    & = d_p\varphi(b) = d_p\varphi(y)^q = d_p\varphi(y)^{\hat q} \\
    & = d_{\hat q}\tilde \varphi(x)^p = d_{\hat q}\tilde \varphi(x)^{p'} \\
    & = d_{p'}\varphi(y)^{\hat q} = d_{p'}\varphi(y)^q = d_{p'}\varphi(b) \\
    & = \varphi(x,b)^{p'}.
  \end{align*}
  Therefore $p = p'$.
\end{proof}

Given $A \subseteq M \models T$ and $p \in \tS_\varphi(\acl(A))$, we denote the unique
$\acl(A)$-definable extension of $p$ to $M$ by $p\rest^M$.
The definition of $p\rest^M$ is an $\acl(A)$-definable
$\tilde \varphi$-predicate which does not depend on $M$, and we may
therefore refer to it unambiguously as $d_p\varphi(y)$ (so far we only
used the notation $d_p\varphi(y)$ when $p$ was a $\varphi$-type over a model).

If $p \in \tS_\varphi(A)$, we define (with some abuse of notation)
$p\rest^M
= \{q\rest^M\colon p \subseteq q \in \tS_\varphi(\acl(A))\}$.

\begin{prp}
  If $A \subseteq M \models T$ and $p \in \tS_\varphi(A)$ as above, then
  $p\rest^M$ is closed in $\tS_\varphi(M)$.

  Assume moreover that $M$ is $(|A|+\omega)^+$-saturated and strongly
  homogeneous, and let $P = \{q \in \tS_\varphi(M)\colon p \subseteq q\}$.
  Then $p\rest^M$ is the unique $A$-good set contained in $P$.

  Also, we have $p\rest^M = \bigcap_{\varepsilon>0} \CBm_{\varphi,M,\varepsilon}(P)$, and
  in fact $p\rest^M = \bigcap_{\varepsilon\in E} \CBm_{\varphi,M,\varepsilon}(P)$ for any
  $E \subseteq (0,\infty)$ such that $\inf E = 0$.
  In other words, $q \in p\rest^M$ if and only if $q \in P$,
  and it has maximal $\CB_{\varphi,M,\varepsilon}$-rank as such for every $\varepsilon > 0$.
\end{prp}
\begin{proof}
  If $M \preceq M'$, then
  $p\rest^M = \{q\rest_M\colon q \in p\rest^{M'}\}$,
  so we may assume that $M$ is
  $(|A|+\omega)^+$-saturated and strongly homogeneous.

  Let $Q \subseteq P$ be any $A$-good subset.
  If $q \in Q$, then
  $q = (q\rest_{\acl(A)})\rest^M \in p\rest^M$.
  By \fref{lem:LocAclTrans} it follows that
  $Q = p\rest^M$.
  Therefore $p\rest^M$ is closed.

  It follows by \fref{lem:goodexist} that
  $p\rest^M \subseteq \CBm_{\varphi,M,\varepsilon}(P)$
  for every $\varepsilon > 0$, whereby $\bigcap_{\varepsilon > 0} \CBm_{\varphi,M,\varepsilon}(P) \neq \emptyset$.
  The other requirements for $\bigcap_{\varepsilon > 0} \CBm_{\varphi,M,\varepsilon}(P)$ to be
  $A$-good follow directly from its definition, and we conclude that
  $p\rest^M = \bigcap_{\varepsilon > 0} \CBm_{\varphi,M,\varepsilon}(P)$.
\end{proof}

\begin{prp}
  Assume that $M \preceq N \models T$ are both $\omega$-saturated.
  Let $p \in \tS_\varphi(M)$, and let $q \in \tS_\varphi(N)$ extend it.
  Then $\CB_{\varphi,M,\varepsilon}(p) \geq \CB_{\varphi,N,\varepsilon}(q)$, and equality holds for all
  $\varepsilon > 0$ if and only if $q = p\rest^N$.
\end{prp}
\begin{proof}
  Assume first that $N$ is $\omega_1$-saturated and strongly homogeneous.
  Let:
  \begin{align*}
    X_{\varepsilon,\alpha} = \{p \in \tS_\varphi(M)\colon \CB_{\varphi,M,\varepsilon}(p) \geq \alpha\} & &
    Y_{\varepsilon,\alpha} = \{q \in \tS_\varphi(N)\colon \CB_{\varphi,N,\varepsilon}(q) \geq \alpha\}
  \end{align*}

  We first prove by induction on $\alpha$ that if $\CB_{\varphi,M,\varepsilon}(p) \leq \alpha$ and
  $p \subseteq q \in \tS_\varphi(N)$ then $\CB_{\varphi,N,\varepsilon}(q) \leq \alpha$.
  Given a $\varphi$-predicate $\psi(x,a)$ with parameters $a \in M$ and
  $r \in [0,1]$, let us write:
  \begin{gather*}
    [\psi(x,a) < r]^M
    = [\psi(x,a) < r]^{\tS_\varphi(M)}
    = \{p' \in \tS_\varphi(M)\colon \psi(x,a)^{p'} < r\}.
  \end{gather*}
  Sets of this form form a basis of open sets for $\tS_\varphi(M)$.
  Since $\CB_{\varphi,M,\varepsilon}(p) \leq \alpha$, there are such $\psi(x,a)$ and $r$ such
  that $p \in [\psi(x,a)<r]^M$ and $\diam([\psi(x,a) < r]^M\cap X_{\varepsilon,\alpha}) \leq \varepsilon$.

  Clearly, $q \in [\psi(x,a) < r]^N$, so we'll be done if we prove that
  $\diam([\psi(x,a) < r]^N\cap Y_{\varepsilon,\alpha}) \leq \varepsilon$ as well.
  Indeed, assume that there are $q',q'' \in [\psi(x,a)<r]^N$ such that
  $d(q',q'') > \varepsilon$.
  Let $d\varphi(y,e')$ and $d\varphi(y,e'')$ be their respective definitions,
  where $e'$ and $e''$ are the canonical parameters.
  Then we can find $f',f'',b \in M$ such that $f'f''b \equiv e'e''a$ and
  $d(a,b)$ is as small as we want (in fact using $\omega$-saturation we can
  actually have $a = b$, but the argument goes through even if we
  can only have $b$ arbitrarily close to $a$; therefore the result
  is true even if $M$ is merely \emph{approximately $\omega$-saturated}, as
  defined in \cite{BenYaacov:Morley} or \cite{BenYaacov-Usvyatsov:dFiniteness}).

  Let $p',p'' \in \tS_\varphi(M)$ be defined by $d\varphi(y,f')$ and $d\varphi(y,f'')$,
  respectively.
  Then $p',p'' \in [\psi(x,b)<r]^M$, and having made sure that $b$ is
  close enough to $a$, we can get $p',p'' \in [\psi(x,a)<r]^M$.
  Also, we still have $d(p',p'') > \varepsilon$.
  Therefore at least one of $p' \notin X_{\varepsilon,\alpha}$ or $p'' \notin X_{\varepsilon,\alpha}$ must
  hold, so let's say it's the former.
  In this case, by the induction hypothesis,
  $p'\rest^N \notin Y_{\varepsilon,\alpha}$.
  Since $N$ is $\omega_1$-strongly homogeneous, $p'\rest^N$ and $q'$
  are conjugates by $\Aut(N)$, so $q' \notin Y_{\varepsilon,\alpha}$, as required.

  Now let $q = p\rest^N$, and let $d\varphi(y,e)$ be the common
  definition.
  Assume that $\CB_{\varphi,N,\varepsilon}(q) \leq \alpha$, so there are a $\varphi$-predicate
  $\psi(x,a)$ with $a \in N$, and $r$, such that $q \in [\psi(x,a)<r]^N$ and
  $\diam([\psi(x,a)<r]^N\cap Y_{\alpha,\varepsilon}) \leq \varepsilon$.
  We may find $b,f \in M$ such that $bf \equiv ae$, and such that $f$ is as
  close as we want to $e$ (again: if we take the plain definition of
  $\omega$-saturation we can even have $e=f$, but we want an argument that
  goes through if $M$ is only approximately $\omega$-saturated).
  If $p' \in \tS_\varphi(M)$ is defined by $d\varphi(y,f)$,
  then $p' \in [\psi(x,b)<r]^M$, and assuming $f$ and $e$
  are close enough we also
  have $p \in [\psi(x,b)<r]^M$.
  By the homogeneity assumption for $N$ we get
  $\diam([\psi(x,b)<r]^N\cap Y_{\alpha,\varepsilon}) \leq \varepsilon$.

  Assume now that $p'',p''' \in [\psi(x,b)<r]^M$, and $d(p'',p''') > \varepsilon$.
  Let $q'' = p''\rest^N$ and
  $q''' = p'''\rest^N$.
  Then $d(q'',q''') > \varepsilon$, so either $q'' \notin Y_{\varepsilon,\alpha}$ or
  $q''' \notin Y_{\varepsilon,\alpha}$ (or both), so let's say it's the former.
  By the induction hypothesis we get $p'' \notin X_{\varepsilon,\alpha}$, showing that
  $\diam([\psi(x,b)<r]^M\cap X_{\alpha,\varepsilon}) \leq \varepsilon$.
  Therefore $\CB_{\varphi,M,\varepsilon}(p) \leq \alpha$.

  Since $p\rest^N$ is the unique extension of $p$ to $N$
  having maximal $\CB_{\varphi,N,\varepsilon}$ rank for all $\varepsilon > 0$, this proves what
  we wanted.

  If $N$ is not $\omega_1$-strongly homogeneous, take a common elementary
  extension for $M$ and $N$ which is.
\end{proof}

\begin{cor}
  Again, let $M \preceq N$ be $\omega$-saturated.
  Let $X \subseteq \tS_\varphi(M)$ be any set of $\varphi$-types (without any
  further assumptions), and let $Y \subseteq \tS_\varphi(N)$ be its pre-image under
  the restriction mapping.
  Then $\CB_{\varphi,M,\varepsilon}(X) = \CB_{\varphi,N,\varepsilon}(Y)$.
\end{cor}

In particular, this gives us an absolute notion of $\CB_{\varphi,\varepsilon}(p)$
where $p$ is a partial $\varphi$-type, without specifying over which
($\omega$-saturated) model we work: just calculate it in any $\omega$-saturated
model containing the parameters for $p$.

For example, assume that $A \subseteq B \subseteq M \models T$, $p \in \tS_\varphi(A)$ and $p \subseteq q \in
\tS_\varphi(B)$.
Then $q$ is definable over $\acl(A)$ if and only if $\CB_{\varphi,\varepsilon}(p) =
\CB_{\varphi,\varepsilon}(q)$ for all $\varepsilon > 0$.

\section{Global stability and independence}
\label{sec:GlobStab}

In the previous section we only considered local stability, i.e.,
stability of a single formula $\varphi(x,y)$.
In this section we will use those results to deduce a global stability
theory.

\subsection{Gluing local types}

Let $A$ be an algebraically closed set, and say $A \subseteq M$.
Let $\varphi(x,y)$ and $\psi(x,z)$ be two stable formulae,
$p_\varphi \in \tS_\varphi(A)$, $p_\psi \in \tS_\psi(A)$.
Then we know that each of $p_\varphi$ and $p_\psi$ have unique extensions
$q_\varphi \in \tS_\varphi(M)$ and $q_\psi \in \tS_\psi(M)$, respectively, which are
$A$-definable.

Assume now that $p_\varphi$ and $p_\psi$ are
compatible, i.e., that $p_\varphi(x)\cup p_\psi(x)$ is satisfiable.
We would like to show that $q_\varphi$ and $q_\psi$ are compatible as well.
For this purpose there is no harm in assuming that $M$ is strongly
$(|A|+\omega)^+$-homogeneous, or even that $M = \fM$ is our monster
model.

Let $t,w$ be any variables in a single sort, say the home sort,
and $e \neq e' \in M$ in that sort.
We may assume that $d(e,e') = 1$: even if not, everything we do below
would work when we replace $d(t,w)$ with
$d(t,w) \dotplus \ldots \dotplus d(t,w)$.
Define:
\begin{gather*}
  \chi_{\varphi,\psi}(x,yztw) = \varphi(x,y)\land d(t,w) \dotplus \psi(x,z)\land \lnot d(t,w).
\end{gather*}

Since we assume that $\varphi$ and $\psi$ are stable
so is $\chi_{\varphi,\psi}(x,yztw)$ by the following easy result:
\begin{lem}
  Assume $\varphi_i(x,y)$ are stable formulae for $i<n$ and $f$ is an
  $n$-ary continuous connective.
  Then $(f\circ\varphi_{<n})(x,y)$ is stable as well.
\end{lem}

Let $a$, $b$ and $c$ be in the appropriate sorts.
Then:
\begin{align*}
  \varphi(a,b) & = \chi_{\varphi,\psi}(a,bcee'), \\
  \psi(a,c) & = \chi_{\varphi,\psi}(a,bcee).
\end{align*}
Thus every instance of $\varphi$ and $\psi$ is an instance of $\chi_{\varphi,\psi}$,
so every $\varphi$-predicate or $\psi$-predicate (with parameters in $M$) is a
$\chi_{\varphi,\psi}$-predicate.
Moreover, if $B \subseteq M$ and $\rho(x)$ is $B$-definable as a
$\varphi$-predicate (or $\psi$-predicate) then it is $B$-invariant and
therefore a $B$-definable $\chi_{\varphi,\psi}$-predicate.
Notice that there is no need to assume here that $e,e' \in B$.

We therefore obtain for every set $B \subseteq M$ quotient mappings
$\theta_\varphi\colon \tS_{\chi_{\varphi,\psi}}(B) \to \tS_\varphi(B)$ and
$\theta_\psi\colon \tS_{\chi_{\varphi,\psi}}(B) \to \tS_\psi(B)$, and if $B \subseteq C$ then the following
diagram commutes:
\begin{gather*}
  \xymatrix{
    \tS_{\chi_{\varphi,\psi}}(C) \ar[rr] \ar[d] & & \tS_\varphi(C) \times \tS_\psi(C) \ar[d] \\
    \tS_{\chi_{\varphi,\psi}}(B) \ar[rr]        & & \tS_\varphi(B) \times \tS_\psi(B)
  }
\end{gather*}

Let us now return to the situation we started with, namely $A \subseteq M$
algebraically closed and a pair of compatible $p_\varphi \in \tS_\varphi(A)$ and
$p_\psi \in \tS_\psi(A)$.
Since they are compatible, there exists
$p_{\chi_{\varphi,\psi}} \in \tS_{\chi_{\varphi,\psi}}(A)$ such that
$p_\varphi = \theta_\varphi(p_{\chi_{\varphi,\psi}})$, $p_\psi = \theta_\psi(p_{\chi_{\varphi,\psi}})$ ($p_\varphi$ and $p_\psi$
actually determine $p_{\chi_{\varphi,\psi}}$, but we do not need this fact).

Let $q_{\chi_{\varphi,\psi}} \in \tS_{\chi_{\varphi,\psi}}(M)$ be the unique extension of
$p_{\chi_{\varphi,\psi}}$ which is $A$-definable, and let
$q_\varphi' = \theta_\varphi(q_{\chi_{\varphi,\psi}}) \in \tS_\varphi(M)$, $q_\psi' = \theta_\psi(q_{\chi_{\varphi,\psi}})$.
Then $q_\varphi'$ is definable over $M$ (by stability of $\varphi$) and
is invariant under $\Aut(M/A)$ (since $q_{\chi_{\varphi,\psi}}$ is).
Since $M$ is strongly $(|A|+\omega)^+$-homogeneous, it follows that the
definition $d_{q_\varphi'}\varphi(y)$ is $A$-invariant, and therefore
$A$-definable.
By the commutativity of the diagram above (with $B = A$, $C = M$)
$q'_\varphi$ extends $p_\varphi$.
Therefore, by uniqueness, $q_\varphi = q'_\varphi$.
Similarly, $q_\psi = q'_\psi$, so $q_\varphi \cup q_\psi \subseteq q_{\chi_{\varphi,\psi}}$ is satisfiable.

\subsection{Global stability}

\begin{dfn}
  A theory $T$ is stable if all formulae are stable in $T$.
\end{dfn}

\begin{dfn}
  A theory $T$ is $\lambda$-stable if for all $n < \omega$ and all
  sets $A$ such that $|A| \leq \lambda$:
  $\|\tS_n(A)\| \leq \lambda$.
\end{dfn}

\begin{dfn}
  Let $M$ be a model and $p \in \tS_n(M)$.
  We say that $p$ is \emph{definable} if $p\rest_\varphi$ is for every
  formula of the form $\varphi(x_{<n},\bar y)$, i.e., if for every such
  formula there is an $M$-definable predicate $d_p\varphi(\bar y)$, called
  the \emph{$\varphi$-definition} of $p$, such that for all $\bar b \in M$:
  $$\varphi(\bar x,\bar b)^p = d_p\varphi(\bar b).$$
\end{dfn}

\begin{thm}
  The following are equivalent for a theory $T$:
  \begin{enumerate}
  \item $T$ is stable.
  \item All types over models are definable.
  \item $T$ is $\lambda$-stable for all $\lambda$ such that $\lambda = \lambda^{|T|}$.
  \item $T$ is $\lambda$-stable for some $\lambda \geq |T|$.
  \end{enumerate}
\end{thm}
\begin{proof}
  \begin{cycprf}
  \item[\eqnum{1}{2}] By \fref{prp:LocStabCnt}.
  \item[\impnum{1}{3}] Assume $T$ is stable $\lambda = \lambda^{|T|}$, $n < \omega$, and
    $|A| \leq \lambda$.
    Then by Downward L\"owenheim-Skolem we can find $M \supseteq A$ such that
    $\|M\| \leq \lambda$.
    Let $\varphi(x_{<n},\bar y)$ be any formula.
    Then by \fref{prp:LocStabCnt} we have
    $\|\tS_\varphi(M)\| \leq \lambda$, whereby $|\tS_\varphi(M)| \leq \lambda^\omega = \lambda$.
    Let $\{\varphi_i(x_{<n},\bar y)\colon i < |T|\}$ enumerate all formulae of this
    form.
    Then $|\tS_n(M)| \leq \prod |\tS_{\varphi_i}(M)| \leq \lambda^{|T|} = \lambda$, and \emph{a fortiori}
    $\|\tS_n(M)\| \leq \lambda$.
  \item[\impnext] Let $\lambda = 2^{|T|}$.
  \item[\impfirst] Let $T$ be $\lambda$-stable ($\lambda \geq |T|$) and
    $\varphi(x_{<n},\bar y)$ be any formula, and we will
    show that $\varphi$ is stable in $T$.
    Let $M$ be any model such that $\|M\| \leq \lambda$, so let $M_0 \subseteq M$ be a
    dense subset such that $|M_0| = \lambda$.
    Then $\tS_n(M) = \tS_n(M_0)$ (i.e., the quotient mapping
    $\tS_n(M) \to \tS_n(M_0)$ is a homeomorphism and an isometry).
    Since $\varphi$ is uniformly continuous, the quotient
    mapping $\tS_n(M) \to \tS_\varphi(M)$ is uniformly continuous as a mapping
    between metric spaces.
    Therefore $\|\tS_\varphi(M)\| \leq \|\tS_n(M)\| = \|\tS_n(M_0)\| \leq \lambda$.

    Since this holds for all $M$ such that $\|M\| \leq \lambda$, we conclude
    by \fref{prp:LocStabCnt} that $\varphi$ is stable.
  \end{cycprf}
\end{proof}

\begin{conv}
  From now on we assume $T$ is stable.
\end{conv}

\begin{prp}
  \label{prp:GlobStat}
  Let $A \subseteq M$, where $A$ is algebraically closed, and let
  $p(x) \in \tS_x(A)$.
  Then $p$ has a unique extension to $M$, denoted
  $p\rest^M$, which is $A$-definable.
  Moreover, the $A$-definable definitions of such extensions do
  not depend on $M$, and will be denoted as usual by $d_p\varphi$.
\end{prp}
\begin{proof}
  For every formula $\varphi(x,y)$ let $d_p\varphi = d_{p\rest_\varphi} \varphi$.
  Then uniqueness and moreover part are already a consequence of
  \fref{prp:LocStat}.
  Thus all that is left to show is that the following set of
  conditions is satisfiable (and therefore a complete type):
  \begin{align*}
    p\rest^M
    & = \{\varphi(x,b) = d_p\varphi(x,b)^M\colon \varphi(x,y) \in \cL,
    b \in M \text{ in the sort of } y\} \\
    & = \bigcup_{\varphi(x,y) \in \cL} (p\rest_\varphi)\rest^M.
  \end{align*}
  (Here $x$ is fixed but $y$ varies with $\varphi$.)

  By compactness it suffices to show this for unions over finitely
  many formulae $\varphi$.
  For two formulae this was proved is the previous subsection, by
  coding both formulae in a single one.
  But we can repeat this process encoding any finite set of formulae
  in a single one, whence the required result.
\end{proof}

\begin{dfn}
  \label{dfn:NonFork}
  Let $A \subseteq B$, $p \in \tS(B)$.
  We say that $p$ \emph{does not fork} over $A$ if there exists an
  extension $p \subseteq q \in \tS_n(\acl(B))$ such that all the definitions
  $d_q\varphi$ are over $\acl(A)$.

  If $\bar a$ is a tuple, $A$ and $B$ sets, and $\tp(\bar a/AB)$ does
  not fork over $A$, we say $\bar a$ is \emph{independent} from $B$
  over $A$, in symbols $\bar a \ind_A B$.
\end{dfn}

\begin{cor}
  \label{cor:GlobStat}
  Let $A \subseteq B$, where $A$ is algebraically closed, and let
  $p \in \tS_n(A)$.
  Then there exists a unique $q \in \tS_n(B)$ extending $p$ and
  non-forking over $A$.
  This unique non-forking extension is denoted $p\rest^B$, and is
  given explicitly as
  $$p\rest^B = \{\varphi(\bar x,\bar b) = d_p\varphi(\bar b)\colon
  \varphi(\bar x,\bar y) \in \cL, \bar b \in B\}.$$
\end{cor}
\begin{proof}
  Let $M$ be any model such that $B \subseteq M$.
  Then $\acl(B) \subseteq M$, and $(p\rest^M)\rest_{\acl(B)} = p\rest^{\acl(B)}$ is
  $A$-definable, so $p\rest^B$ is a non-forking extension of $p$.

  Conversely, let $q \in \tS_n(B)$ be a non-forking extension of $p$.
  Then there exists $q' \in \tS_n(\acl(B))$ which is $A$-definable.
  Then ${q'}\rest^M$ is an $A$-definable extension of $p$, so
  ${q'}\rest^M = p\rest^M$, whereby $q = p\rest^B$.
\end{proof}

\begin{cor}
  \label{cor:NonForkSameDef}
  Let $A$ and $B$ be sets, $\bar a$ a tuple, and let
  $p = \tp(\bar a/\acl(A))$ and $q = \tp(\bar a/\acl(AB))$.
  Then $\bar a \ind_A B$ if and only if $d_p\varphi = d_q\varphi$ for every
  formula $\varphi(\bar x,\bar y)$.
\end{cor}
\begin{proof}
  Right to left is immediate from the definition.
  So assume $\bar a \ind_A B$.
  This means there is a type $q' \in \tS_n(\acl(AB))$ extending
  $\tp(\bar a/AB)$ such that $d_{q'}\varphi$ is $\acl(A)$-definable for all
  $\varphi$.
  Then $q$ and $q'$ are conjugates by an automorphism fixing $AB$.
  Such an automorphism would fix $\acl(A)$ setwise, so
  $d_q\varphi$ is $\acl(A)$-definable for all $\varphi$.
  Now let $M \supseteq AB$ be a model, and $r = q\rest^M$.
  Then $d_q\varphi = d_r\varphi$ by definition, and $r$ is $\acl(A)$-definable and
  extends $p$, whereby $d_p\varphi = d_r\varphi$.
\end{proof}

We conclude:
\begin{thm}
  Assume $T$ is stable.
  Then:
  \begin{enumerate}
  \item \emph{Invariance:} The relation $\ind$ is
    automorphism-invariant.
  \item \emph{Symmetry:}
    $\bar a \ind_A \bar b \Longleftrightarrow \bar b \ind_A \bar a$
  \item \emph{Transitivity:}
    $\bar a \ind_A BC$ if and only if
    $\bar a \ind_A B$ and $\bar a \ind_{AB} C$.
  \item \emph{Existence:}
    For all $\bar a$, $A$ and $B$ there is $\bar b \equiv_A \bar a$ such
    that $\bar b \ind_A B$.
  \item \emph{Finite character:}
    $\bar a \ind_A B$ if and only if $\bar a \ind_A \bar b$ for all
    finite tuples $\bar b \in B$.
  \item \emph{Local character:}
    For all $\bar a$ and $A$ there is $A_0 \subseteq A$ such that
    $|A_0| \leq |T|$ and $\bar a \ind_{A_0} A$.
  \item \emph{Stationarity:}
    Assume $A$ is algebraically closed, and $B \supseteq A$.
    If $\bar a \equiv_A \bar b$ and $\bar a \ind_A B$, $\bar b \ind_A B$
    then $\bar a \equiv_{AB} \bar b$.
  \end{enumerate}
\end{thm}
\begin{proof}
  Invariance is clear.

  Symmetry follows from \fref{prp:LocSym}.

  Transitivity is immediate from \fref{cor:NonForkSameDef}.

  For existence, we may replace $B$ with any model containing $B$.
  Then let $p'$ be any extension of $p = \tp(\bar a/A)$ to $\acl(A)$,
  and then let $\bar b$ realise the unique non-forking extension of
  $p'$ to $M$.

  For finite character, we may replace $A$ with $\acl(A)$ without
  changing the statement.
  But then $\bar a \ind_A B$ if and only if
  $\tp(\bar a/AB) = \tp(\bar a/A)\rest^{AB}$, and if this fails it is
  due to some finite tuple in $B$.

  Let $p = \tp(\bar a/\acl(A))$.
  Recall we defined $\Cb_\varphi(p)$ as the canonical parameter
  of $d_p\varphi$.
  Let $\Cb(p) = \{\Cb_\varphi(p)\colon \varphi(\bar x,\ldots) \in \cL\}$ be the \emph{canonical
    base} of $p$.
  Then $|\Cb(p)| = |T|$, and $p$ is definable over its canonical base
  so $\bar a \ind_{\Cb(p)} \acl(A)$.
  For each $c \in \Cb(p)$ we know that $\tp(c/A)$ is algebraic, and going
  back to the definition of algebraicity in \fref{lem:AlgTyp} we
  see that countably many parameters in $A$ suffice to witness this.
  Let $A_0$ be the union of all these witness sets for all
  $c \in \Cb(p)$.
  Then $A_0 \subseteq A$, $|A_0| \leq |T|$, and $\Cb(p) \subseteq \acl(A_0)$, so
  $\bar a \ind_{\acl(A_0)} \acl(A)$, or equivalently,
  $\bar a \ind_{A_0} A$.

  Stationarity is just \fref{cor:GlobStat}.
\end{proof}

\appendix

\section{A remark on continuity moduli on bounded spaces}
\label{apdx:InvContMod}

The usual definition of (uniform) continuity in the metric setting
goes ``for all $\varepsilon > 0$ there is $\delta > 0$ such that\ldots'', whence our
definition of a continuity modulus as a function $\delta\colon (0,\infty) \to (0,\infty)$,
mapping each $\varepsilon$ to a corresponding $\delta$.
We would like to present here an alternative definition, which rather
goes the other way around.

In the cases which interest us all metric spaces (structures and type
spaces) are bounded, usually of diameter $\leq 1$.
We may therefore allow ourselves the following simplification:
\begin{conv}
  Hereafter, all metric spaces are bounded of diameter $\leq 1$.
\end{conv}

\begin{dfn}
  An \emph{inverse continuity modulus} is a continuous monotone
  function $\fu \colon [0,1] \to [0,1]$ such that $\fu(0) = 0$.

  We say that a mapping $f\colon (X,d) \to (X',d')$ \emph{respects} $\fu$, or
  that it is \emph{uniformly continuous with respect to $\fu$}, if for
  every $x,y \in X$:
  $$d'(f(x),f(y)) \leq \fu(d(x,y)).$$
\end{dfn}
In other words, an inverse uniform continuous modulus maps a $\delta$ to an
$\varepsilon$.
(In case that the destination space is not bounded we may still
consider inverse continuity moduli, but then we need to
allow the range of $\fu$ to be $[0,\infty]$.)

\begin{lem}
  Let $\fu$ be an inverse continuity modulus.
  For $\varepsilon > 0$ define $\delta(\varepsilon) = \sup \{t \in [0,1]\colon \fu(t) \leq \varepsilon\}$.
  Then $\delta$ is a continuity modulus, and every function which
  respects $\fu$ (as an inverse uniform continuity modulus) respects
  $\delta$ (as a uniform continuity modulus).
\end{lem}
\begin{proof}
  That $\varepsilon > 0 \Longrightarrow \delta(\varepsilon) > 0$ follows from the fact that $\fu(t) \to 0$ as
  $t \to 0$.
  Assume now that $f\colon (X,d) \to (X',d')$ respects $\fu$, $\varepsilon > 0$, and
  $d(x,y) < \delta(\varepsilon)$.
  By monotonicity of $\fu$ and definition of $\delta$:
  $$d'(f(x),f(y)) \leq \fu(d(x,y)) \leq \varepsilon.$$
\end{proof}

The converse is not much more difficult:
\begin{lem}
  Let $\delta$ be a continuity modulus.
  Then there exists an inverse continuity modulus $\fu$ such that
  every function respecting $\fu$ respects $\delta$.
\end{lem}
\begin{proof}
  For $r,r' \in [0,1]$ define:
  \begin{align*}
    \fu_0(r) & = \inf \{\varepsilon > 0\colon \delta(\varepsilon) > r\} && (\inf \emptyset = 1) \\
    \fu_1(r,r') & =
    \begin{cases}
      \fu_0(r') & r \geq r'\\
      \fu_0(r')\cdot\left( \frac{2r}{r'} - 1 \right)
      & \frac{r'}{2} \leq r < r'\\
      0 & r < \frac{r'}{2}
    \end{cases} \\
    \fu(r) & = \sup_{r' \in [0,1]} \fu_1(r,r').
  \end{align*}
  Then $\fu_0 \colon \setR^+ \to \setR^+$ is an increasing function, not
  necessarily continuous.
  It is however continuous at $0$:
  $\lim_{r\to0^+} \fu_0(r) = 0 = \fu_0(0)$
  (since for every $\varepsilon > 0$, $\fu_0(\delta(\varepsilon)) \leq \varepsilon$).
  For every $r_0 > 0$, the family of function $r \mapsto \fu_1(r,r')$,
  indexed by $r'$, is equally continuous on $[r_0,1]$, so
  $\fu$ is continuous on $(0,1]$.
  We also have $r \leq 1/2 \Longrightarrow \fu(r) \leq \fu_0(2r)$ (since for $r' \geq 2r$,
  $\fu_1(r,r')$ contributes nothing to $\fu(r)$),
  whereby $\lim_{r\to0^+} \fu(r) = 0 = \fu(0)$.
  Therefore $\fu$ is continuous on $[0,1]$, and therefore an inverse
  continuity modulus.
  Observe also that:
  $$\fu(r) \geq \fu_1(r,r) = \fu_0(r)
  \geq \sup \{\varepsilon \leq 1\colon (\forall 0 < \varepsilon' < \varepsilon)(\delta(\varepsilon') \leq r)\}.$$

  Assume now that $f\colon (X,d) \to (X',d')$ respects $\delta$.
  If $x,y \in X$ and $\varepsilon > 0$ satisfy $d'(f(x),f(y)) > \varepsilon$,
  Then $d(x,y) \geq \delta(\varepsilon')$ for all $0 < \varepsilon' < \varepsilon$, whereby
  $\fu(d(x,y)) \geq \varepsilon$.
  Therefore $d'(f(x),f(y)) \leq \fu(d(x,y))$, and $f$ respects $\fu$, as
  required.
\end{proof}

Together we obtain:
\begin{thm}
  A mapping between bounded metric spaces $f\colon (X,d) \to (X',d')$ is
  uniformly continuous with respect to a (standard) continuity modulus
  if and only if it is uniformly continuous with respect to an inverse
  one.
  In other words, the two distinct definitions of continuity moduli
  give rise to the same notion of uniform continuity.
\end{thm}

Inverse continuity moduli give us (continuously) a direct answer to
the question ``how much can the value of $f$ change from $x$ to $y$?''
For example, if we attached to symbols in a signature inverse
continuity moduli, rather than usual ones, the axiom scheme
\fref{ax:uc} would take the more elegant form:

\begin{align*}
  \tag{UC$^\fu_\cL$}
  & \begin{aligned}
    & \sup_{x_{<i},y_{<n-i-1},z,w}
    d(f(\bar x,z,\bar y),f(\bar x,w,\bar y))\dotminus
    \fu_{f,i}(d(z,w))
    = 0 \\
    & \sup_{x_{<i},y_{<n-i-1},z,w}
    |P(\bar x,z,\bar y)-P(\bar x,w,\bar y)|\dotminus
    \fu_{P,i}(d(z,w))
    = 0.
  \end{aligned}
\end{align*}

Since the continuity moduli $\fu_{s,i}$ are continuous functions, they
can be admitted as connectives in the language, so the above can
indeed be viewed as sentences.
In fact, in almost all actual cases, the inverse continuity moduli can
be directly constructed from the standard connectives
$\{\lnot,\dotminus,\frac{x}{2}\}$, so there is no need to introduce new
connectives (for example, in the case of probability algebras, all the
inverse continuity moduli can be taken to be the identity).

\section{On stability inside a model}
\label{apdx:ModStab}

This appendix answers a question posed by C.\ Ward Henson to the first
author concerning stability of a formula inside a specific structure
(in contrast with stability of a formula in all models of a theory,
discussed in \fref{sec:LocStab} above).
The notion of stability inside a model appears for example in the work
of Krivine and Maurey on stable Banach spaces
\cite{Krivine-Maurey:EspacesDeBanachStables}:
a Banach space $E$ is stable in the sense of Krivine and Maurey
precisely if the formula
$\|x + y\|$ is stable in the unit ball of $E$ (viewed as a continuous
structure in an appropriate language) in the sense defined below.

\begin{dfn}
  Let $M$ be a structure, $\varphi(x,y)$ be a formula and $\varepsilon > 0$.
  We say that $\varphi$ is \emph{$\varepsilon$-stable in $M$} if there is no sequence
  $(a_ib_i\colon i < \omega)$ in $M$ such that $|\varphi(a_i,b_j) - \varphi(a_j,b_i)| \geq \varepsilon$
  for all $i < j < \omega$.
  We say that $\varphi$ is \emph{stable in $M$} if it is $\varepsilon$-stable in $M$
  for all $\varepsilon > 0$.
\end{dfn}

Note that $\varphi(x,y)$ is ($\varepsilon$-)stable in $M$ if and only if
$\tilde \varphi(y,x)$ is.
Also, $\varphi$ is \hbox{($\varepsilon$-)stable} in a theory $T$ if and only if it is
in every model of $T$.

\begin{lem}
  Assume that $\varphi(x,y)$ is $\varepsilon$-stable in $M$.
  Then for every $p \in \tS_\varphi(M)$ there exists a finite sequence
  $(c_i\colon i < n)$ in $M$ such that for all $a,b \in M$:
  \begin{align*}
    (\forall i < n) (\varphi(c_i,a) \leq \varphi(c_i,b) + \varepsilon)
    \Longrightarrow
    \varphi(x,a)^p \leq \varphi(x,b)^p + 3\varepsilon.
  \end{align*}
\end{lem}
\begin{proof}
  Assume not.
  We will choose by induction on $n$ elements $a_n,b_n,c_n \in M$ and
  $r_n,s_n \in [0,1]$ as follows.
  At each step, there are by assumption $a_n,b_n \in M$ such that
  $\varphi(c_i,a_n) \leq \varphi(c_i,b_n) + \varepsilon$ for all $i < n$, and yet
  $\varphi(x,a_n)^p > \varphi(x,b_n)^p + 3\varepsilon$.
  Choose $r_n,s_n$ such that
  $\varphi(x,b_n)^p < r_n < r_n + 3\varepsilon < s_n < \varphi(x,a_n)^p$.
  Once these choices are made we have $\varphi(x,a_i)^p > s_i$ and
  $\varphi(x,b_i)^p < r_i$ for all $i \leq n$, and we may therefore find
  $c_n \in M$ such that $\varphi(c_n,a_i) > s_i$ and
  $\varphi(c_n,b_i) < r_i$ for all $i \leq n$.

  Once the construction is complete, for every $i < j < \omega$
  colour the pair $\{i,j\}$ as follows: if $s_i > \varphi(c_i,a_j) + \varepsilon$,
  colour the pair $\{i,j\}$ yellowish maroon; otherwise, colour it
  fluorescent pink.
  Notice that if $\{i,j\}$ is fluorescent pink then
  $\varphi(c_i,b_j) - \varepsilon > r_i$.
  By Ramsey's Theorem there is an infinite monochromatic subset
  $I \subseteq \omega$, and without loss of generality $I = \omega$.
  If all pairs are fluorescent pink then we have for all $i < j < \omega$:
  $\varphi(c_j,a_i) - \varphi(c_i,a_j) > s_i - (s_i-\varepsilon) = \varepsilon$.
  If all are yellowish maroon we get
  $\varphi(c_i,b_j) - \varphi(c_j,b_i) > (r_i+\varepsilon)-r_i = \varepsilon$.
  Either way, we get a contradiction to $\varepsilon$-stability in $M$.
\end{proof}

\begin{lem}
  Assume that $\varphi(x,y)$ is $\varepsilon$-stable in $M$.
  Then for every $p \in \tS_\varphi(M)$ there exists a finite sequence
  $(c_i\colon i < n)$ in $M$ and a continuous increasing function
  $f\colon [0,1]^n \to [0,1]$, such that for all $a \in M$:
  \begin{align*}
    |\varphi(x,a)^p - f(\varphi(c_i,a)\colon i < n)| \leq 3\varepsilon.
  \end{align*}
\end{lem}
\begin{proof}
  Let $(c_i\colon i < n)$ be chosen as in the previous Lemma.
  As a first approximation, let:
  \begin{align*}
    f(\bar u) = \sup \{ \varphi(x,a)^p\colon a \in M \text{ and } \varphi(c_i,a) \leq u_i
    \text{ for all } i < n\}.
  \end{align*}
  This function is increasing, but not necessarily continuous.
  We define a family of auxiliary functions
  $h_{\bar u}\colon [0,1]^n \to [0,1]$ for $\bar u \in [0,1]^n$:
  \begin{gather*}
    h_{\bar u}(\bar v) = \frac{1}{\varepsilon}\bigwedge_{i<n} (((v_i + \varepsilon) \dotminus u_i) \land \varepsilon).
  \end{gather*}
  In other words, $h_{\bar u}(\bar v)$ is a linear-by-pieces function
  which is equal to $1$ if $v_i \geq u_i$ for all $i$, to $0$ if
  $v_i \leq u_i - \varepsilon$ for some $i$, and is linear in between.
  We now define:
  \begin{gather*}
    g(\bar v) = \sup_{\bar u \in [0,1]^n} h_{\bar u}(\bar v)f(\bar u).
  \end{gather*}
  Since all the functions $h_{\bar u}$ are are equally continuous, $g$
  is continuous.
  It is clearly increasing, and also satisfies for all
  $\bar v \in [0,1]$:
  \begin{gather*}
    f(\bar v) \leq g(\bar v) \leq f(\bar v \dotplus \varepsilon).
  \end{gather*}
  Indeed, the first inequality is clear from the definition, and the
  second follows from the fact that $f$ is increasing, and every tuple
  $\bar u$ such that $f(\bar u)$ contributes to $g(\bar v)$ must be
  smaller in every coordinate than $\bar v \dotplus \varepsilon$.

  Now let $a \in M$ and $v_i = \varphi(c_i,a)$ for $i < n$.
  Then by choice of $\bar c$:
  \begin{align*}
    g(\bar v)
    & \leq f(\bar v + \varepsilon) \\
    & = \sup \{ \varphi(x,b)^p\colon b \in M \text{ and } \varphi(c_i,b) \leq \varphi(c_i,a) + \varepsilon
    \text{ for all } i < n\} \\
    & \leq \varphi(x,a)^p + 3\varepsilon,
    \intertext{and}
    g(\bar v)
    & \geq f(\bar v) \\
    & =  \sup \{ \varphi(x,b)^p\colon b \in M \text{ and } \varphi(c_i,b) \leq \varphi(c_i,a)
    \text{ for all } i < n\} \\
    & \geq \varphi(x,a)^p.
  \end{align*}

  Therefore $|g(\varphi(c_i,a)\colon i < n) - \varphi(x,a)^p| \leq 3\varepsilon$.
\end{proof}

\begin{thm}
  Assume $\varphi$ is stable in $M$.
  Then every $p \in \tS_\varphi(M)$ is definable.
  Moreover, for every such $p$ there is a sequence
  $(c_i\colon i < \omega)$ and a continuous increasing function
  $f\colon [0,1]^\omega > [0,1]$ such that $d_p\varphi(y) = f \circ (\varphi(c_i,y)\colon i < \omega)$.
\end{thm}
\begin{proof}
  For all $m < \omega$ choose a sequence $(c_{m,i}\colon i < n_m)$ and function
  $f_m\colon [0,1]^{n_m} \to [0,1]$ as in the previous Lemma corresponding
  to $\varepsilon = 2^{-m-17}$.
  Let:
  \begin{align*}
    N_m & = \sum_{k < m} n_k, \\
    d_{N_m+i} & = c_{m,i} && i < n_m, \\
    f(u_{<\omega}) & = \flim_m f_m(u_{N_m},\ldots,u_{N_{m+1}-1}).
  \end{align*}
  Since each $f_m$ is increasing and continuous, as is
  $\flim\colon [0,1]^\omega \to[0,1]$, we have that $f$ is increasing and
  continuous.
  Also, by the choice of the parameters we have for all $a \in M$:
  \begin{align*}
    f(\varphi(d_i,a)\colon i < \omega)
    & = \flim_m f_m (\varphi(c_i,a)\colon i < n_m) = \varphi(x,a)^p.
    \qedhere
  \end{align*}
\end{proof}

Notice that we get almost the same result as for a formula which is
stable in a theory: the definition is still a limit of positive (i.e.,
increasing) continuous combinations of instances of $\varphi$ with
parameters in $M$.
However, these combinations are not necessarily the
particularly elegant median value as in \fref{sec:LocStab}.

\bibliographystyle{amsalpha}
\bibliography{mybib}

\end{document}